\documentclass[DIV=14,letterpaper]{scrartcl}

\date{}

\usepackage{tikz}
\usetikzlibrary{calc}
\usetikzlibrary{matrix}
\usepackage{amsmath,amssymb,amsfonts,amsthm,subfig}
\usepackage[pdftex]{hyperref} 

\usepackage{graphicx}
\allowdisplaybreaks

\newcommand{\V}[1]{\mbox{\boldmath $ #1 $}}

\newcommand{\bey}{\begin{eqnarray}}
\newcommand{\eey}{\end{eqnarray}}

\newcommand{\beq}{\begin{equation}}
\newcommand{\eeq}{\end{equation}}
\theoremstyle{plain}
\newtheorem{thm}{\hspace{6mm}Theorem}[section]

\newtheorem{lem}{\hspace{6mm}Lemma}[section]

\theoremstyle{definition}

\theoremstyle{remark}

\newtheorem{rem}{\hspace{6mm}Remark}[section]

\numberwithin{equation}{section}
\numberwithin{figure}{section}

\def\bD{\mathbb{D}}
\def\bM{\mathbb{M}}
\def\Th{\mathcal{T}_h}

\newcommand{\cT}{{\mathcal T}}

\title{A study on anisotropic mesh adaptation for finite element approximation of eigenvalue problems
with anisotropic diffusion operators}

\author{Jingyue Wang
\thanks{Department of Mathematics, the University of Kansas, Lawrence, KS 66045, U.S.A. 
({\em jwang@math.ku.edu}).}
\and Weizhang~Huang
\thanks{Department of Mathematics, the University of Kansas, Lawrence, KS 66045, U.S.A. 
({\em whuang@ku.edu}).}
}

\begin{document}
\vskip 1cm
\maketitle

\begin{abstract}
Anisotropic mesh adaptation is studied for the linear finite element solution of eigenvalue problems
with anisotropic diffusion operators. The $\bM$-uniform mesh approach is employed with which any nonuniform mesh
is characterized mathematically as a uniform one in the metric specified by a metric tensor.
Bounds on the error in the computed eigenvalues
are established for quasi-$\bM$-uniform meshes. Numerical examples arising from
the Laplace-Beltrami operator on parameterized surfaces, nonlinear diffusion,
thermal diffusion in a magnetic field in plasma physics, and the Laplacian operator on an L-shape domain
are presented. Numerical results show that anisotropic adaptive meshes can lead to more accurate computed
eigenvalues than uniform or isotropic adaptive meshes. They also confirm the second order convergence of the error
that is predicted by the theoretical analysis.

The effects of approximation of curved boundaries on the computation of eigenvalue problems
is also studied in two dimensions. It is shown that the initial mesh used to define the geometry of the physical domain
should contain at least $\sqrt{N}$ boundary points to keep
the effects of boundary approximation at the level of the error of the finite element approximation,
where $N$ is the number of the elements in the final adaptive mesh.
Only about $N^{1/3}$ boundary points in the initial mesh 
are needed for boundary value problems. This implies that the computation of
eigenvalue problems is more sensitive to the boundary approximation than that of boundary value problems.
\end{abstract}

\noindent{\bf AMS 2010 Mathematics Subject Classification.}
65N15, 65N25, 65N30, 65N50

\noindent{\bf Key Words.}
anisotropic mesh, mesh adaptation, anisotropic diffusion, eigenvalue problem, finite element, error analysis

\noindent{\bf Abbreviated title.}
Anisotropic mesh adaptation for eigenvalue problems

\section{Introduction}

We are concerned with the adaptive linear finite element approximation for the eigenvalue problem
of a general anisotropic diffusion operator, 
\begin{equation}
\begin{cases}
-\nabla\cdot (\bD \nabla u) = \lambda \, \rho\,  u, & \quad\text{ in }\Omega\\
u = 0, & \quad\text{ on } \partial \Omega
\end{cases}
\label{eigen-1}
\end{equation}
where $\Omega \subset \mathbb{R}^d$ ($d \ge 1$) is a polygonal or polyhedral domain,
$\bD = \bD(\V{x})$ is the diffusion matrix assumed to be symmetric and uniformly positive definite on $\Omega$,
and $\rho(\V{x})$ is a given weight function satisfying
\[
0 < \rho_0 \le \rho(x) \le \rho_1 < + \infty, \quad \int_\Omega \rho(\V{x}) d \V{x} = 1 .
\]
Problem (\ref{eigen-1}) is isotropic when the diffusion matrix has the form $\bD = \alpha(\V{x}) I$
for some scalar function $\alpha(\V{x})$ and anisotropic otherwise.

Anisotropic eigenvalue problems in the form of (\ref{eigen-1}) arise in the application of
the Laplace-Beltrami operator to geometric shape analysis \cite{Reuter-09,Reuter-06}
where the eigenvalues and eigenfunctions
are used to extract shape-DNA, a numerical fingerprint or signature,
of any two or three dimensional surface or solid and provide insights
in the structure and morphology of the shape.
They can also arise from the stability or sensitivity analysis and the construction
of special solutions (such as traveling wave and standing wave solutions)
for anisotropic diffusion partial differential equations (PDEs), which occur in many areas of science and engineering
such as plasma physics \cite{GL09,SH07},
petroleum engineering \cite{ABBM98a,ABBM98b,MD06},
and image processing \cite{PM90, Wei98}.

Numerical solution of eigenvalue problems for elliptic operators, both adaptive and non-adaptive, 
has been studied extensively in the past; e.g., see
Birkhoff et al. \cite{BdBSW1966},
Fix \cite{Fix1973},
Babu{\v{s}}ka and Osborn \cite{BO1989,BO1991},
Xu and Zhou \cite{XZ2001},
Dai et al. \cite{DXZ2008},
Boffi et al. \cite{Bof2010, Boffi2010b},
Mehrmann and Miedlar \cite{MM2011},
and references therein.
However, most of the existing work is concerned with isotropic diffusion problems or
isotropic mesh adaptation. Much less is known about numerical computation
of anisotropic eigenvalue problems and anisotropic mesh adaptation for general eigenvalue problems.
In principle, algorithms and theory developed for general eigenvalue problems
can be applied to anisotropic eigenvalue problems. However, they can be ineffective/inefficient
since they do not take into consideration the interplay between the diffusion directions,
the mesh geometry, and the geometry of eigenfunctions.
Such an interplay has proven important for proper
mesh adaptation and for conditioning of algebraic systems for the finite
element approximation of boundary value problems (BVPs) and initial-boundary value problems
(e.g., see \cite{Hua05b,HuKaLa2013,KaHu2013b,KaHuXu2012}).
Moreover, a standard finite element or finite difference discretization of anisotropic eigenvalue problems
can result in a discrete eigenvalue problem exhibiting totally different structures \cite{Huang2013},
a phenomenon that rarely occurs with problems of isotropic diffusion.


The objective of this paper is to study anisotropic mesh adaptation for use in the linear
finite element approximation of anisotropic eigenvalue problems in the form (\ref{eigen-1}).
The starting point of the study is a classic result of Raviart and Thomas \cite{Raviart-Thomas-83}
stating that the error in the finite element approximation of the eigenvalues is bounded
by the squared energy norm of the error of projection from the corresponding eigenfunction
space into the finite element space. The anisotropic mesh adaptation is based on a bound
for the projection error. We use the so-called $\bM$-uniform mesh approach \cite{Hua06,HR11}
for anisotropic mesh adaptation. Within the approach, any nonuniform mesh is viewed as
a uniform one in the metric specified by a tensor. Using the mathematical characterizations
of the meshes, the metric tensor and the error bound for eigenvalues are obtained by minimizing
a bound on the project error over quasi-$\bM$-uniform meshes. The main theoretical result is stated in
Theorem~\ref{thm-err}. We present numerical results obtained for four numerical examples arising from
the Laplace-Beltrami operator on parameterized surfaces, nonlinear diffusion,
thermal diffusion in a magnetic field in plasma physics, and the Laplacian operator on an L-shape domain.
These numerical results show that anisotropic adaptive meshes can lead to more accurate computed
eigenvalues than uniform or isotropic adaptive meshes. They also demonstrate
the second order convergence of the error predicted by the theoretical analysis.

The effects of approximation of curved boundaries of the physical domain on the computation of eigenvalue problems
is also studied. It is shown that about $N^{1/2}$ boundary points in the initial
mesh used to define the geometry of the domain, where $N$ is the number of the elements
in the final adaptive mesh, are needed to keep
the effects of boundary approximation at the level of the error of the finite element approximation.
On the other hand, only about $N^{1/3}$ boundary points are needed in the initial mesh for BVPs.
This indicates that the computation of eigenvalue problems is more sensitive to the boundary approximation
than that of BVPs.

An outline of the paper is as follows.  In \S\ref{SEC:fem}, the linear finite element approximation of (\ref{eigen-1})
is described and several properties of the discrete and continuous problems are stated.
Anisotropic mesh adaptation for the eigenvalue problem is discussed and a bound on the error
in approximated eigenvalues is obtained in \S\ref{SEC:ama}. Numerical results are presented
in \S\ref{SEC:examples}. Effects of the approximation of curved boundaries are studied in \S\ref{SEC:boundary}.
Finally, \S\ref{SEC:conclusion} contains conclusions.

\section{Finite Element Formulation}
\label{SEC:fem}

In this section we describe the linear finite element approximation for the eigenvalue problem (\ref{eigen-1})
and state some basic properties of (\ref{eigen-1}) and the corresponding discrete eigenvalue problem.

The variational formulation of (\ref{eigen-1}) is to find $\lambda \in \mathbb{R}$ and
$0 \not \equiv u \in H_{0}^{1}(\Omega)$ such that
\begin{equation}
a(u,v) =\lambda\; b(u,v), \qquad\forall v\in H_0^1(\Omega)
\label{var-1}
\end{equation}
where
\[
a(u,v) =\int_{\Omega}(\mathbb{D}\nabla u)\cdot\nabla v d \V{x},\quad
b(u, v) = \int_{\Omega} \rho u v d \V{x} .
\]
It is known that the bilinear form $a(\cdot,\cdot)$: $H_0^1(\Omega) \times H_0^1(\Omega)\to\mathbb{R}$
is bounded, coercive, and symmetric. Moreover,
(\ref{var-1}) has a countable set of eigenvalues (e.g., see D'yakonov \cite{Dyakonov1996}),
\[
0 < \lambda_{1}\le\lambda_{2}\le\cdots \le \lambda_k \le \cdots ,
\]
where the multiplicities of the eigenvalues have been taken into account.
Furthermore, the corresponding eigenfunctions (normalized in the norm induced by $b(\cdot, \cdot)$),
\[
 u_{1},\, u_{2}, \, \cdots, \, u_{k}, \, \cdots, 
 \]
form a complete orthonormal basis for $H_0^1(\Omega)$. Finally,
the eigenvalue problem satisfies Poincar{\'e}'s min-max principle,
\[
\lambda_{k}=\min_{\text{dim}\left(E\right)=k}\max_{v\in E}\frac{a\left(v,v\right)}{b\left(v,v\right)},
\quad k = 1, 2, ...
\]
where $E$ denotes a subset of $H_0^1(\Omega)$.

To define the linear finite element approximation, we assume that an affine family of
simplicial triangulations, $\{\mathcal{T}_{h}\}$, is given for $\Omega$.
Let $V^{h}\subset H_0^1(\Omega)$ be the linear finite element space associated with $\mathcal{T}_{h}$.
Then, the linear finite element approximation of (\ref{eigen-1}) is to find
$\lambda^h \in \mathbb{R}$ and $0 \not\equiv u^{h}\in V^{h}$ such that
\begin{equation}
a(u^{h},v^{h})=\lambda^{h}\, b(u^{h},v^{h}), \quad\forall\, v^{h}\in V^{h}.
\label{fem-1}
\end{equation}
Expressing the linear finite element space into
$V^h = \text{span}\{ \phi_1, ..., \phi_{N_{v}} \}$,
where $\phi_{j}$ ($j=1, ..., N_v$) is the linear basis function associated with $j^{\text{th}}$ interior vertex
and $N_v$ is the number
of the interior vertices of $\Th$, we can rewrite (\ref{fem-1}) in a matrix form as
\beq
A \V{u} = \lambda^h\; M \V{u},
\label{fem-2}
\eeq
where $\V{u} = [u_1, ..., u_{N_v}]^T$ is the vector associated
with $u^h = \sum_j u_j \phi_j$ and the stiffness and mass matrices
$A$ and $M$ are given by
\[
a_{i,j} = a(\phi_j, \phi_i), \quad m_{i,j} = b(\phi_j, \phi_i) , \quad i, j = 1, ..., N_v .
\]

Denote by $E_k$ the linear space spanned by the first $k$ eigenfunctions, i.e.,
\[
E_k = \text{span}\{ u_1, ..., u_k\} .
\]
We have the following classic result for the finite element approximation of eigenvalue problems.

\begin{lem}[Raviart and Thomas \cite{Raviart-Thomas-83}]
\label{lem:RT83}
Assume that the eigenvalues of the finite element eigenvalue problem (\ref{fem-1})
are expressed in a non-decreasing order as
\[
0 < \lambda_{1}^h\le\lambda_{2}^h\le\cdots\le\lambda_{N_v}^h .
\]
For any given integer $k$ $(1\le k<N_v)$, 
\begin{equation}
 0 \le \frac{\lambda_j^h - \lambda_j}{\lambda_j^h} \le
 C(k)\sup_{\substack{v\in E_{k} \\ \|v\|_\rho =1} } \|v-\Pi_{h}v\|_E^{2},\qquad1\le j\le k
\label{eigenvalue-bound}
\end{equation}
where $\| \cdot \|_E = \sqrt{a(\cdot,\cdot)}$ is the energy norm, $\| \cdot \|_\rho = \sqrt{b(\cdot, \cdot)}$
is the weighted $L^2$ norm (with weight $\rho$),
$\Pi_{h}$ is the elliptic projection operator from $H_0^{1}(\Omega)$ to $V^{h}$ under the energy norm,
and $C(k)$ is a constant depending only on $\Omega$ and $k$.
\end{lem}

The above lemma shows that the error in approximating the eigenvalues $\lambda_j$, $j = 1, ..., k$
is bounded by the error of projection the corresponding eigenfunction space $E_k$ into the linear finite element
space $V^h$.  The lemma serves as the starting point for anisotropic mesh adaptation (to minimize a bound
on the projection error), which is studied in the next section.

\section{Anisotropic mesh adaptation}
\label{SEC:ama}

In this section we study anisotropic mesh adaptation for the linear finite element approximation
(\ref{fem-1}). The key is to estimate the error of projection from the eigenfunction space $E_k$
to the finite element space $V^h$, as shown in Lemma~\ref{lem:RT83}.
We use the so-called $\bM$-uniform mesh approach \cite{Hua06,HR11} for anisotropic mesh adaptation
with which any nonuniform mesh is characterized mathematically as a uniform one in the metric specified
by a metric tensor, which is in turn determined by minimizing a bound on the projection error
over quasi-$\bM$-uniform meshes.

To start with, we introduce some notation. Denote the number of the elements of mesh $\Th$ by $N$.
We assume that the reference element $\hat{K}$ is taken to be equilateral
with an unitary volume.  The coordinates on $K$ and $\hat{K}$ are denoted by $\V{x}$ and $\V{\xi}$, respectively.
For any element $K\in \Th$, the affine mapping from $\hat{K}$
to $K$ is denoted by $F_K$ and its Jacobian matrix by $F_K'$. 
Notice that $F_K'$ is a constant matrix and $\det(F_K') = |K|$, where
$|K|$ denotes the volume of $K$.
We assume for the moment that the metric tensor $\bM = \bM(\V{x})$ is given.
(A definition of $\bM$ will be given in Theorem \ref{thm-err}.)
It is known \cite{Hua06,HR11} that an $\bM$-uniform mesh
associated with $\bM$ satisfies the so-called equidistribution and alignment conditions,
\begin{align}
& |K|\det(\bM_{K})^{\frac{1}{2}} = \frac{\sigma_{h}}{N}, \quad \forall\, K \in \Th
\label{eq-1}
\\
& \frac{1}{d}\text{tr}((F_{K}')^{T}\bM_{K}F_{K}') = \det((F_{K}')^{T}\bM_{K}F_{K}')^{\frac{1}{d}}
\quad \left \{ = \left ( |K|\;\det(\bM_{K})^{\frac{1}{2}}\right )^{\frac{2}{d}} \right \} ,\quad
\forall\, K \in \Th
\label{ali-1}
\end{align}
where $\det(\cdot)$ and $\text{tr}(\cdot)$ denote the determinant and trace of a matrix, respectively,
and
\[
\sigma_{h}=\sum_{K}|K| \det(\bM_{K})^{\frac{1}{2}}, \quad
\bM_{K}=\frac{1}{|K|}\int_{K}\bM(\V{x})\, d \V{x}.
\]
The left-hand side of (\ref{eq-1}), $|K|\det(\bM_{K})^{\frac{1}{2}}$, is the volume of
$K$ in the metric $\bM_K$ and $\sigma_h$ is the sum of the volume of all elements of
$\Th$ measured in the metric specified by $\{ \bM_K,\, K \in \Th \}$, the piecewise constant
approximation of $\bM$. (Without confusion, hereafter the metric specified by this
piecewise constant approximation of $\bM$ will simply be referred to as the metric $\bM$.)
Thus, the equidistribution condition (\ref{eq-1}) implies that all of the elements have the same volume
in the metric $\bM$.
Moreover, it is not difficult to show that the left-hand side of (\ref{ali-1}), $\frac{1}{d}\text{tr}((F_{K}')^{T}\bM_{K}F_{K}')$,
is equivalent to the squared diameter of $K$ in the metric $\bM_K$, while the right-hand side is the squared
average diameter of $K$ in the same metric. Thus, the alignment condition (\ref{ali-1}) implies that
$K$ is equilateral in the metric $\bM_K$.
Furthermore, (\ref{ali-1}) implies (e.g., see \cite{Hua06,HR11}) that $K$ is aligned
with $\bM_K$ in the sense that the principal directions of the circumscribed ellipsoid of $K$
coincide with the directions of the eigenvectors of $\bM_K$ and the lengths of the principal axes
of the ellipsoid are inversely proportional to the square root of the eigenvalues.

It should be pointed out that (\ref{eq-1}) and (\ref{ali-1}) are highly nonlinear and in actual computation,
it is hard, if not impossible,
to obtain a mesh to satisfy them exactly. Thus, it is more practical to use meshes satisfying relaxed conditions.
We consider quasi-$\bM$-uniform meshes which satisfy the inequalities
\begin{align}
& |K|\;\det(\bM_{K})^{\frac{1}{2}} \le C_{eq}\, \frac{\sigma_{h}}{N}, \quad \forall\, K \in \Th
\label{eq-2}
\\
& \frac{1}{d}\text{tr}((F_{K}')^{T}\bM_{K}F_{K}') \le C_{ali}\, \det((F_{K}')^{T}\bM_{K}F_{K}')^{\frac{1}{d}}
= C_{ali} |K|^{\frac 2 d} \det(\bM_{K})^{\frac{1}{d}},\quad \forall\, K \in \Th
\label{ali-2}
\end{align}
for some small or modest positive constants $C_{eq}$ and $C_{ali}$.
It is easy to show that (\ref{eq-2}) and (\ref{ali-2}) reduce to (\ref{eq-1}) and (\ref{ali-1})
when $C_{eq}=C_{ali}=1$. We notice that 
\[
\|  (F_{K}')^{T}\bM_{K}F_{K}' \|_2 \le \text{tr}((F_{K}')^{T}\bM_{K}F_{K}') \le d \; \|  (F_{K}')^{T}\bM_{K}F_{K}' \|_2 ,
\]
where $\| \cdot \|_2$ is the matrix 2-norm. Thus, the 2-norm and trace of a symmetric and positive
semi-definite matrix can be used equivalently. Moreover, (\ref{ali-2}) implies (e.g., see \cite[(15)]{KaHu2013}) that
\begin{equation}
\frac{1}{d}\text{tr}((F_{K}')^{-1}\bM_{K}^{-1} (F_{K}')^{-T} ) \le \left (\frac{d \, C_{ali}}{d-1}\right )^{d-1}
\, |K|^{- \frac 2 d} \det(\bM_{K})^{- \frac{1}{d}},\quad
\forall\, K \in \Th .
\label{ali-2+1}
\end{equation}

Having obtained the mathematical characterization of quasi-$\bM$-uniform meshes, we now consider the Hessian of
the eigenfunctions. We denote the Hessian of eigenfunction $u_j$ ($1 \le j \le k$) by $H(u_j)(\V{x})$ and
let $| H(u_j) |(\V{x}) = \sqrt{H(u_j)^2(\V{x})}$. For a given positive integer $k$, we assume that we can choose
a symmetric and uniformly positive semi-definite matrix-valued function $H = H(\V{x})$ such that
\begin{equation}
| H(u_j) |(\V{x}) \le H(\V{x}), \quad \V{x} \in \Omega,  \quad j = 1, ..., k
\label{H-1}
\end{equation}
where the sign ``$\le$'' is in the sense of negative semi-definiteness.
Inequality (\ref{H-1}) implies
\begin{equation}
\text{tr}(A^T | H(u_j) |(\V{x}) A) \le \text{tr}(A^T H(\V{x}) A),\quad\forall\, A \in \mathbb{R}^{d\times d},
\quad j = 1, ..., k .
\label{H-2}
\end{equation}
This follows from
\[
\text{tr}(A^T H(\V{x}) A) - \text{tr}(A^T | H(u_j) |(\V{x}) A) = \sum_i \V{a}_i^T (H(\V{x}) - | H(u_j) |(\V{x})) \V{a}_i
\ge 0,
\]
where $A$ has been expressed as $A = [\V{a}_1, ..., \V{a}_d]$.

There are several ways to choose $H$. A simple choice is $H(\V{x}) = \sum_{j=1}^k | H(u_j) |(\V{x})$.
Another one is the sequential intersection of $| H(u_j) |(\V{x}),\, j = 1, ..., k$. To explain the latter, we recall that
any pair of symmetric and positive definite matrices $A$ and $B$ can be simultaneously diagonalized as
\[
X^T A X = \text{diag}(\sigma_i), \quad X^T B X = I ,
\]
where $X$ is a nonsingular matrix. Then, the intersection of $A$ and $B$ is defined as
\[
A \cap B = X^{-T} \text{diag}(\max\{1, \sigma_i\}) X^{-1} .
\] 
(The word intersection is referred to the intersection of the ellipsoids defined by $A$ and $B$.)
Obviously, $A\cap B$ is symmetric and positive definite and satisfies
\[
A \le A\cap B, \quad B \le A\cap B .
\]
If $| H(u_j) |(\V{x}),\, j = 1, ..., k$, are not singular, $H(\V{x})$ can be defined as
\begin{equation}
H(\V{x}) = ((\cdots ((| H(u_1) | \cap | H(u_2) |) \cap | H(u_3) |) \cap \cdots) \cap | H(u_k) |) .
\label{H-3}
\end{equation}
This choice is more accurate than the first one. It is optimal when $k=2$ but not when $k \ge 3$.
We use this choice in our computation.

We now are ready to establish an error bound for the projection from $E_k$ to $V^h$.

\begin{thm}
\label{thm-err}
For a given positive integer $k$, we assume that the Hessian matrices, $H(u_j)$ $(1 \le j \le k)$, are uniformly
positive definite on $\Omega$. If the metric tensor is taken as
\begin{equation}
\bM_K =
\det(H_K)^{-\frac{1}{d+2}} \max\limits_{\V{x} \in K} \| H_K \bD(\V{x}) \|_2^{\frac{2}{d+2}}
 \left ( \frac{1}{|K|}\| H_K^{-1} H \|_{L^2(K)}^{2} \right )^{\frac{2}{d+2}} \cdot  H_K,
\quad \forall K \in \Th
\label{M-1}
\end{equation}
where $H_K$ is the average of $H$ over $K$,
and if the mesh satisfies the conditions (\ref{eq-2}) and (\ref{ali-2}),
then the error for the projection from $E_k$ to $V^h$ and error in the computed eigenvalues are bounded by
\begin{align}
& \sup_{\substack{v\in E_{k} \\ \|v\|_\rho = 1} } \|v-\Pi_h v \|_{E}^2
\le C N^{-\frac{2}{d}} C_{ali}^{d+1} C_{eq}^{\frac{2}{d}} \sigma_h^{\frac{d+2}{d}} ,
\label{error-3} \\
& 0 \le \frac{\lambda_j^{h}-\lambda_j}{\lambda_j^{h}}
\le  C N^{-\frac{2}{d}} C_{ali}^{d+1}C_{eq}^{\frac{2+d}{ d}}\sigma_{h}^{\frac{d+2}{d}},
\quad j = 1, ..., k
\label{thm-err-1}
\end{align}
where $C$ is a constant depending only on $k$ and 
\begin{equation}
\sigma_{h}= \sum_{K} |K| \det(\bM_K)^{\frac{1}{2}}
=\sum_{K}|K|\, \det(H_{K})^{\frac{1}{d+2}} \max_{x\in K}\|H_{K}\bD (\V{x}) \|_2^{\frac{d}{d+2}}
\left ( \frac{1}{|K|}\| H_K^{-1} H \|_{L^2(K)}^{2} \right )^{\frac{d}{d+2}} .
\label{sigma-1}
\end{equation}
\end{thm}

\begin{proof}
We need only to prove  (\ref{error-3}) since (\ref{thm-err-1}) follows directly from Lemma~\ref{lem:RT83} and (\ref{error-3}).

For any $v \in H_0^1(\Omega)$, by the definition of elliptic projection we have
\[
a(v -\Pi_h v ,w^{h})=0,\quad\forall w^{h}\in V^{h}.
\]
From this and the Cauchy-Schwarz inequality,
\[
\|v-\Pi_h v \|_{E}^{2}=a(v-\Pi_h v ,v-\Pi_h v)=a(v-\Pi_h v,v-I_{h}v)
\le \|v-\Pi_h \|_{E} \|v-I_{h}v\|_{E},
\]
where $I_{h}$ is the interpolation operator associated
with the finite element space $V^{h}$. Thus, we have 
\begin{equation}
\|v-\Pi_h v \|_{E}\le\|v-I_h v \|_{E}, \quad \forall v \in H_0^1(\Omega) .
\label{interp-1}
\end{equation}

Notice that any $v \in E_k$ with $\| v \|_\rho = 1$ can be expressed as
\[
v = \sum\limits_{j=1}^{k} \beta_j u_j \quad \text{ with }\quad \sum\limits_{j=1}^k \beta_j^2 = 1.
\]
From (\ref{interp-1}), the triangle inequality, and the Cauchy-Schwarz inequality, we have
\begin{align*}
\|v-\Pi_h v \|_{E}^2 &\le \|v-I_h v \|_{E}^2  = \|\sum_{j} \beta_j (u_j - I_h u_j) \|_{E}^2\\
& \le \left ( \sum_{j} |\beta_j| \; \| u_j - I_h u_j\|_E \right )^2 \\
& \le \sum_j \beta_j^2 \cdot \sum_j  \|u_j - I_h u_j \|_{E}^2 \\
& = \sum_j  \|u_j - I_h u_j \|_{E}^2 ,
\end{align*}
which implies
\begin{equation}
\sup_{\substack{v\in E_{k} \\ \|v\|_\rho =1} } \|v-\Pi_h v \|_{E}^2
\le \sum_j  \|u_j - I_h u_j \|_{E}^2 .
\label{error-1}
\end{equation}

If we denote the gradient operator with respect to the coordinate $\V{\xi}$ on $\hat{K}$ by
$\hat{\nabla}$, from the chain rule we have $\nabla=(F_{K}')^{-T}\hat{\nabla}$. Then, by changing variables
we get, for $ 1 \le j \le k$,
\begin{align*}
\|u_j - I_h u_j \|_{E}^2 & = \sum_{K} \int_K (\nabla (u_j - I_h u_j))^T \bD(\V{x}) \nabla (u_j - I_h u_j) d \V{x} \\
& = \sum_{K}|K| \int_{\hat{K}}\hat{\nabla}(u_j-I_{h}u_j)^{T}
(F_{K}')^{-1}\mathbb{D}(\V{x}) (F_{K}')^{-T}\hat{\nabla}(u_j-I_{h}u_j) d\V{\xi}\\
 & \le\sum_{K}|K|\;\|\hat{\nabla}(u_j-I_{h}u_j)\|_{L^{2}(\hat{K})}^{2}
 \max_{\V{x}\in K}\|(F_{K}')^{-1}\mathbb{D}(\V{x})(F_{K}')^{-T}\|_{2}.
\end{align*}
 From \cite[Theorem 5.1.1 and Lemma 5.1.5]{HR11}, it follows that 
\begin{align*}
\|\hat{\nabla}(u_j-I_{h}u_j)\|_{L^{2}(\hat{K})}^2 \le C |u_j|_{H^{2}(\hat{K})}^2
\le \frac{C}{|K|}\int_{K}\| (F_{K}')^{T}|H(u_j)|F_{K}'\|_2^{2}d\V{x}.
\end{align*}
Combining the above results with (\ref{error-1}) and using (\ref{H-2}) (and the equivalence between the trace and 2-norm
of symmetric and positive definite matrices), we get
\begin{align*}
& \sup_{\substack{v\in E_{k} \\ \|v\|_\rho =1} } \|v-\Pi_h v \|_{E}^2
\notag \\
& \le C \sum_K  \int_{K}\sum_j \| (F_{K}')^{T}|H(u_j)|F_{K}'\|_2^{2}d\V{x}
\cdot \max_{\V{x}\in K}\|(F_{K}')^{-1}\mathbb{D}(\V{x})(F_{K}')^{-T}\|_{2} \\
& \le C \sum_K \int_{K}\| (F_{K}')^{T}H F_{K}'\|_2^{2}d\V{x} 
\cdot \max_{\V{x}\in K}\|(F_{K}')^{-1}\mathbb{D}(\V{x})(F_{K}')^{-T}\|_{2} ,
\end{align*}
where a factor $k$ has been absorbed into the generic constant $C$. Noticing that
\[
\| (F_{K}')^{T}H F_{K}'\|_2 \le \| (F_{K}')^{T}H_K F_{K}'\|_2  \| H_K^{-1} H \|_2,
\]
we have
\begin{align}
\sup_{\substack{v\in E_{k} \\ \|v\|_\rho =1} } \|v-\Pi_h v \|_{E}^2
\le C \sum_K \| (F_{K}')^{T}H_K F_{K}'\|_2^{2} \cdot \| H_K^{-1} H \|_{L^2(K)}^{2}
\cdot \max_{\V{x}\in K}\|(F_{K}')^{-1}\mathbb{D}(\V{x})(F_{K}')^{-T}\|_{2} .
\label{error-2}
\end{align}

We now use the mesh conditions (\ref{eq-2}) and (\ref{ali-2}) to further develop the above error bound.
The metric tensor (\ref{M-1}) can be rewritten as
\begin{equation}
\bM_K = \theta_K H_K,
\label{M-2}
\end{equation}
where
\begin{equation}
\theta_K = \det(H_K)^{-\frac{1}{d+2}} \max\limits_{\V{x} \in K} \| H_K \bD(\V{x}) \|_2^{\frac{2}{d+2}}
\left ( \frac{1}{|K|}\| H_K^{-1} H \|_{L^2(K)}^{2} \right )^{\frac{2}{d+2}} .
\label{theta-1}
\end{equation}
Then, from (\ref{ali-2}) it follows that
\[
\| (F_{K}')^{T}H_K F_{K}' \|_2 = \frac{1}{\theta_K} \| (F_{K}')^{T}\bM_K F_{K}' \|_2 
\le \frac{1}{\theta_K} d C_{ali} |K|^{\frac{2}{d}} \det(\bM_K)^{\frac{1}{d}} .
\]
Moreover, from (\ref{ali-2+1}) we have
\begin{align*}
\max_{\V{x}\in K}\|(F_{K}')^{-1}\mathbb{D}(\V{x})(F_{K}')^{-T}\|_{2}  & \le 
\max_{\V{x}\in K}\|(F_{K}')^{-1} \bM_K^{-1} (F_{K}')^{-T}\|_{2} \| \bM_K \mathbb{D}(\V{x}) \|_2 \\
& \le d \left (\frac{d\, C_{ali}}{d-1}\right )^{d-1} |K|^{-\frac{2}{d}} \det(\bM_K)^{-\frac{1}{d}} 
\max_{\V{x}\in K}\| \bM_K \mathbb{D}(\V{x}) \|_2 .
\end{align*}
Combining these with (\ref{error-2}) and absorbing $d$ into $C$, we have
\[
\sup_{\substack{v\in E_{k} \\ \|v\|_\rho =1} } \|v-\Pi_h v \|_{E}^2
\le C C_{ali}^{d+1} \sum_K \frac{|K|}{\theta_K^2} \left ( |K| \det(M_K)^{\frac{1}{2}}\right )^{\frac{2}{d}}
\left ( \frac{1}{|K|}\| H_K^{-1} H \|_{L^2(K)}^{2} \right ) \max_{\V{x}\in K}\| \bM_K \mathbb{D}(\V{x}) \|_2  .
\]
From the equidistribution condition (\ref{eq-2}), this gives
\[
\sup_{\substack{v\in E_{k} \\ \|v\|_\rho =1} } \|v-\Pi_h v \|_{E}^2
\le C N^{-\frac{2}{d}} C_{ali}^{d+1} C_{eq}^{\frac{2}{d}} \sigma_h^{\frac{2}{d}}
\sum_K \frac{|K|}{\theta_K^2} \left ( \frac{1}{|K|}\| H_K^{-1} H \|_{L^2(K)}^{2} \right )
\max_{\V{x}\in K}\| \bM_K \mathbb{D}(\V{x}) \|_2  .
\]
Then, inequality  (\ref{error-3})  follows from this and the equality 
\[
\frac{1}{\theta_K^2} \left ( \frac{1}{|K|}\| H_K^{-1} H \|_{L^2(K)}^{2} \right ) \max_{\V{x}\in K}\| \bM_K \mathbb{D}(\V{x}) \|_2
= \det(\bM_K)^{\frac{1}{2}} ,
\]
which can be directly verified using (\ref{M-2}) and (\ref{theta-1}).
\end{proof}

\begin{rem}
\label{rem:3.1}
The assumption that the Hessian matrices, $H(u_j)$ ($ 1 \le j \le k$), are uniformly positive definite on $\Omega$
is not essential. When this assumption is not satisfied,
$|H(u_j)|$ is not guaranteed to be uniformly positive definite on $\Omega$ but remains to be symmetric and
positive semi-definite. In this case, $|H(u_j)|$ can be regularized ($|H(u_j)| \to |H(u_j)| + \alpha_h I $)
with a prescribed small positive constant value of $\alpha_h$ or a more balanced, solution-dependent
choice of $\alpha_h$ (e.g., see \cite{HR11}).
\hfill\qed
\end{rem}

\begin{rem}
\label{rem:3.2}
The right-hand side of (\ref{sigma-1}) is a Riemann sum. When the mesh is sufficiently fine, we can reasonably
expect that
\begin{equation}
\sigma_h \approx \int_\Omega \det(H)^{\frac{1}{d+2}} \| H \bD \|_2^{\frac{d}{d+2}} d \V{x}
= \left \| \sqrt{\det(H)^{\frac{1}{d}} \| H \bD \|_2} \right \|_{L^{\frac{2 d}{d+2}}(\Omega)}^{\frac{2 d}{d+2}} .
\label{sigma-2}
\end{equation}
From this and (\ref{thm-err-1}) we can conclude that $(\lambda_j^h - \lambda_j)/\lambda_j^h = \mathcal{O}(N^{-2/d})$
for $j=1, ..., k$. In words,  the error in the computed eigenvalues decreases at a second order rate
in terms of the average element diameter $\overline{h} = N^{-1/d}$ as $N \to \infty$.

For the case with $\bD = I$, it is well known that the solution dependent factor
in an interpolation error bound is $(\int_\Omega \| H \|_2^2 d \V{x})^{1/2} = \left \| H \right \|_{L^2(\Omega)}$
for a uniform mesh.
From this and Lemma~\ref{lem:RT83} we know that the solution dependent factor in the error bound
for the eigenvalues is $\left \| H \right \|_{L^2(\Omega)}^2$. 
On the other hand, the corresponding factor for a quasi-$\bM$-uniform mesh is
$\sigma_h^{(d+2)/d}$. Since
\[
\sigma_h^{\frac{d+2}{d}} \approx
\left \| \sqrt{\det(H)^{\frac{1}{d}} \| H \bD \|_2} \right \|_{L^{\frac{2 d}{d+2}}(\Omega)}^2
\le \left \| H  \right \|_{L^{\frac{2 d}{d+2}}(\Omega)}^2
\le \left \| H \right \|_{L^2(\Omega)}^2 ,
\]
we can conclude that an adaptive mesh leads to a smaller solution dependent factor in the error bound
for the eigenvalues than a uniform mesh.
\hfill\qed
\end{rem}

\begin{rem}
\label{rem:3.3}
Theorem~\ref{thm-err} reveals how the error is bounded for quasi-$\bM$-uniform meshes. It also shows how
the metric tensor can be defined.
It is unclear whether or not (\ref{M-1}) is optimal,
but it seems to be the best one we can get in the current approach.
\hfill\qed
\end{rem}

\begin{rem}
\label{rem:3.4}
It should be pointed out that the estimate (\ref{thm-err-1}) and the metric tensor (\ref{M-1})
in Theorem~\ref{thm-err} are a priori in the sense that they depend
on the exact eigenfunctions which we are seeking.
In practice, the quasi-$\bM$-uniform meshes associated with (\ref{M-1}) can be generated using an approximate Hessian
recovered from the current computed solution through a Hessian recovery technique.
A number of such techniques are available but known to produce nonconvergent recovered
Hessian from a linear finite element approximation  (e.g., see Kamenski \cite{Kam09}).
Nevertheless, it is shown by Kamenski and Huang \cite{KaHu2013} that
a linear finite element solution of an elliptic BVP converges
at a second order rate as the mesh is refined if the recovered Hessian, which is used
to generate quasi-$\bM$-uniform meshes, satisfies a closeness assumption.
Numerical experiment shows that this closeness assumption is satisfied by 
the approximate Hessian obtained with commonly used Hessian recovery methods.
One of the methods, which is based on least squares fitting, is used in our computation.
\hfill\qed
\end{rem}

\section{Numerical examples}
\label{SEC:examples}

In this section we present numerical results obtained for four two-dimensional examples
to illustrate the error estimate developed in the previous section.
The iterative procedure used in the computation is described in Fig.~\ref{f.1}.
A least squares fitting method is used for Hessian recovery. More specifically, 
a quadratic polynomial is constructed locally for each vertex via least squares fitting to neighboring
nodal function values and then an approximate Hessian at the vertex is obtained
by differentiating the polynomial.
The recovered Hessian is regularized with a prescribed small positive constant $\alpha = 10^{-2}$:
$|H(u_j)| \to |H(u_j)| + \alpha_h I $, $j = 1, ..., k$.
The metric tensor is computed according to (\ref{M-1}) on the current mesh
with $k=4$ (first four eigenvalues) and the Hessian intersection, (\ref{H-3}).
The computed metric tensor and the current mesh (served as the background mesh) are then input into
the Delaunay-type triangular mesh generator BAMG ({B}idimensional {A}nisotropic {M}esh {G}enerator,
developed by Hecht \cite{Hec97}) to generate a corresponding quasi-$\bM$-uniform mesh.
BAMG is known (e.g., see Kamenski and Huang \cite{KaHu2013}) to be capable of producing quasi-$\bM$-uniform
meshes with modest constants $C_{eq}$ and $C_{ali}$.
The algebraic eigenvalue problem (\ref{fem-2}) is solved using the MATLAB large and sparse eigenvalue solver {\tt eigs()}.
The exact eigenvalues of all examples are unavailable. Reference values
used for computing the error are obtained with a very fine adaptive anisotropic mesh.

For comparison purpose, from time to time we include numerical results obtained with {(almost)} uniform meshes
(generated using BAMG with $\bM = I$) and isotropic adaptive meshes. The latter are generated
with a metric tensor based on the minimization of a bound on the $H^1$ semi-norm of linear interpolation error
of eigenfunctions (cf. \cite[(5.190)]{HR11}), i.e.,
\beq
\label{M-3}
\bM_K = \| H_K \|_2^{\frac{4}{d+2}} I ,\quad \forall\, K \in \cT_h .
\eeq

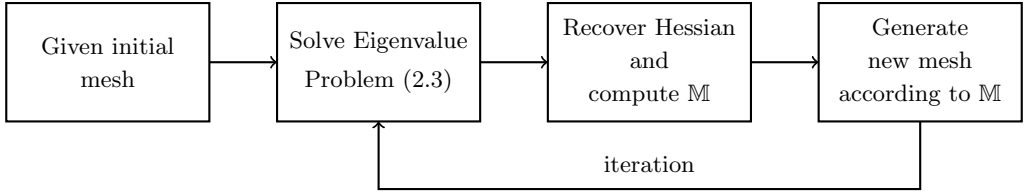
\begin{figure}
\centering 
\tikzset{my node/.code=\ifpgfmatrix\else\tikzset{matrix of nodes}\fi}
\begin{tikzpicture}[every node/.style={my node},scale=0.45] 
{\footnotesize
\draw[thick] (0,0) rectangle (6,3.5);
\draw[thick] (8,0) rectangle (14,3.5);
\draw[thick] (16,0) rectangle (22,3.5);
\draw[thick] (24,0) rectangle (30,3.5);
\draw[->,thick] (6,1.75)--(8,1.75);
\draw[->,thick] (14,1.75)--(16,1.75);
\draw[->,thick] (22,1.75)--(24,1.75);
\draw[->,thick] (27,0)--(27,-2)--(11,-2)--(11,0);
\node[above] at (19,-2) {iteration\\};
\node (node1) at (3,1.75) {Given initial\\mesh\\}; 
\node (node2) at (11,1.75) {Solve Eigenvalue\\ Problem (\ref{fem-2})\\};
\node (node3) at (19,1.75) {Recover Hessian\\  and \\ compute $\bM$\\}; 
\node (node4) at (27,1.75) {Generate\\ new mesh\\ according to $\bM$\\}; 
}
\end{tikzpicture} 
\caption{An iterative procedure for adaptive mesh solution of eigenvalue problem (\ref{eigen-1}).}
\label{f.1}
\end{figure}

\subsection{The Laplace-Beltrami operator on parametrized surfaces}
\label{exam:bunny}

Numerical solution of BVPs of the Laplace-Beltrami operator on surfaces
has been extensively investigated by many researchers; e.g., see Dziuk and Elliott \cite{Dziuk-2013}
and Deckelnick et al. \cite{Dziuk-2005} and references therein. We consider here
numerical solution of eigenvalue problems of the operator on parameterized surfaces using anisotropic mesh adaptation.
This is motivated by its application in geometric shape analysis \cite{Reuter-09,Reuter-06}
where the eigenvalues and eigenfunctions of the Laplace-Beltrami operator
are used to extract shape-DNA of any 2D or 3D manifold (surface or solid)
and provide insights in the structure and morphology of the shape.

Denote a local parametrization of a parameterized surface by $X = X(x,y),\, Y = Y(x,y), \, Z=Z(x, y)$
for $(x, y)\in \Omega$.
Let $\nabla = (\partial/\partial x, \partial/\partial y)^T$. Then the eigenvalue problem of the Laplace-Beltrami operator
reads as
\begin{equation}
\begin{cases}
-\nabla\cdot ( \det(G)^{\frac{1}{2}} G^{-1} \nabla u) = \lambda\,  \det(G)^{\frac{1}{2}} u,
&\quad\text{ in }\Omega \\
u = 0, &\quad \text{ on }\partial\Omega
\end{cases}
\label{LB-1}
\end{equation}
where $G$ is the Riemannian metric tensor on the surface and defined as
\[
G = \left (\frac{\partial (X, Y, Z)}{\partial (x, y)}\right )^T \frac{\partial (X, Y, Z)}{\partial (x, y)}.
\]

We consider an example where the surface is defined by the gray level image of the Stanford bunny
(cf. Fig.~\ref{stanford-bunny-f1:a}). Notice that the photo is not taken directly from a real object. It is a synthesized graph
from a reconstructed polygonal surface with a mesh of 35947 vertices and 69451 triangles. 
We use bilinear interpolation to define values on the surface between grid points. 
Denote the surface by $Z = \psi(x,y)$. We can rewrite (\ref{LB-1})
in the form of (\ref{eigen-1}) with
\begin{equation}
\bD = \frac{1}{(1+\|\nabla \psi \|_2^2)^{\frac{1}{2}}} \left ( \begin{array}{rr}
1 + (\frac{\partial \psi}{\partial y})^2 & - \frac{\partial \psi}{\partial x} \frac{\partial \psi}{\partial y} \\
- \frac{\partial \psi}{\partial x} \frac{\partial \psi}{\partial y}& 1 + (\frac{\partial \psi}{\partial x})^2 \end{array} \right ) ,
\quad \rho = (1+\|\nabla \psi \|_2^2)^{\frac{1}{2}} .
\label{LB-2}
\end{equation}

Fig.~\ref{stanford-bunny-f1:b} shows a typical adaptive mesh obtained with the anisotropic mesh adaptation
method described in the previous section. From the figure one can see that the mesh captures
the detail of the bunny photo very well: mesh elements are concentrated along the boundary and other features
where the gray level changes significantly. 
The relative error (cf. (\ref{thm-err-1})) for the first four eigenvalues is plotted as a function of the number
of elements in Fig.~\ref{fig:error44}. We can see that the convergence history is similar for all of the four eigenvalues.
Interestingly, the asymptotic convergence order (second order or $\mathcal{O}(\bar{h}^2)$,
where $\bar{h} = \sqrt{|\Omega|/N}$)
is not reached until very large $N$. This indicates that the bunny photo requires a huge number
of elements to resolve its complex features. For comparison purpose, we also plot the convergence history
for uniform and isotropic adaptive meshes. It can be seen that anisotropic meshes
produce much more accurate solutions. In particular, for both uniform and isotropic adaptive meshes the error
does not reach its asymptotic convergence order yet at $N \approx 10^{5}$.

The contours of the first four eigenfunctions together with respective color scales are shown in 
Fig.~\ref{fig:ex44-eigenfun}. It is interesting to see that segmentations are induced by the level
sets of the eigenfunctions and capture different regions of the object. This phenomenon has been
observed by Reuter et al. \cite{Reuter-06} where they compute the eigenfunctions directly on surfaces and
use the nodal domains (level sets of eigenfunctions on surface) to analyze surface shapes.

\begin{figure}[thb]
\begin{minipage}{.45\linewidth}
\centering
\subfloat[]{\label{stanford-bunny-f1:a}\includegraphics[scale=.58]{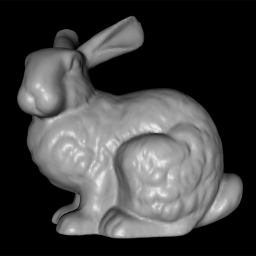}}
\end{minipage}%
\begin{minipage}{.45\linewidth}
\centering
\subfloat[]{\label{stanford-bunny-f1:b}\includegraphics[scale=.8]{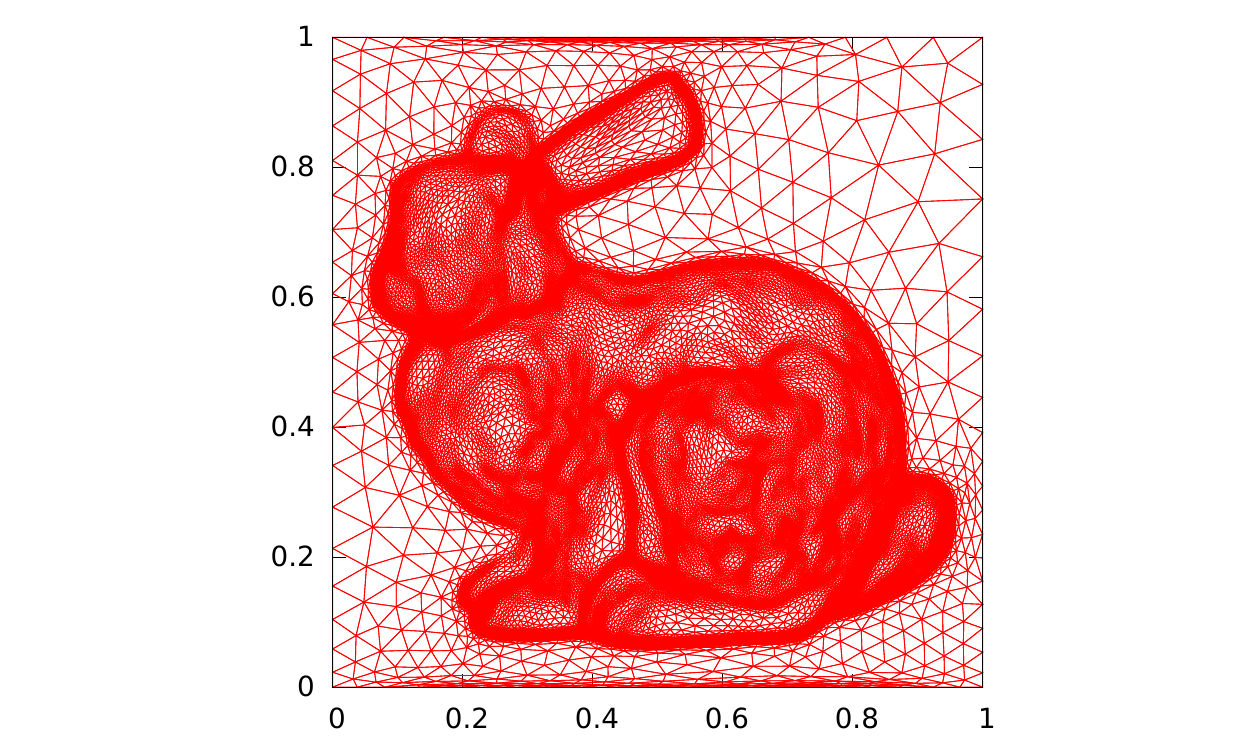}}
\end{minipage}
\caption{For the example in \S\ref{exam:bunny}, an adaptive mesh obtained for the Laplace-Beltrami operator on
the surface defined by the gray level image of the Stanford bunny.}
\label{stanford-bunny-f1}
\end{figure}

\begin{figure}
\begin{minipage}{.5\linewidth}
\centering
\subfloat[]{\label{error44:a}\includegraphics[scale=.9]{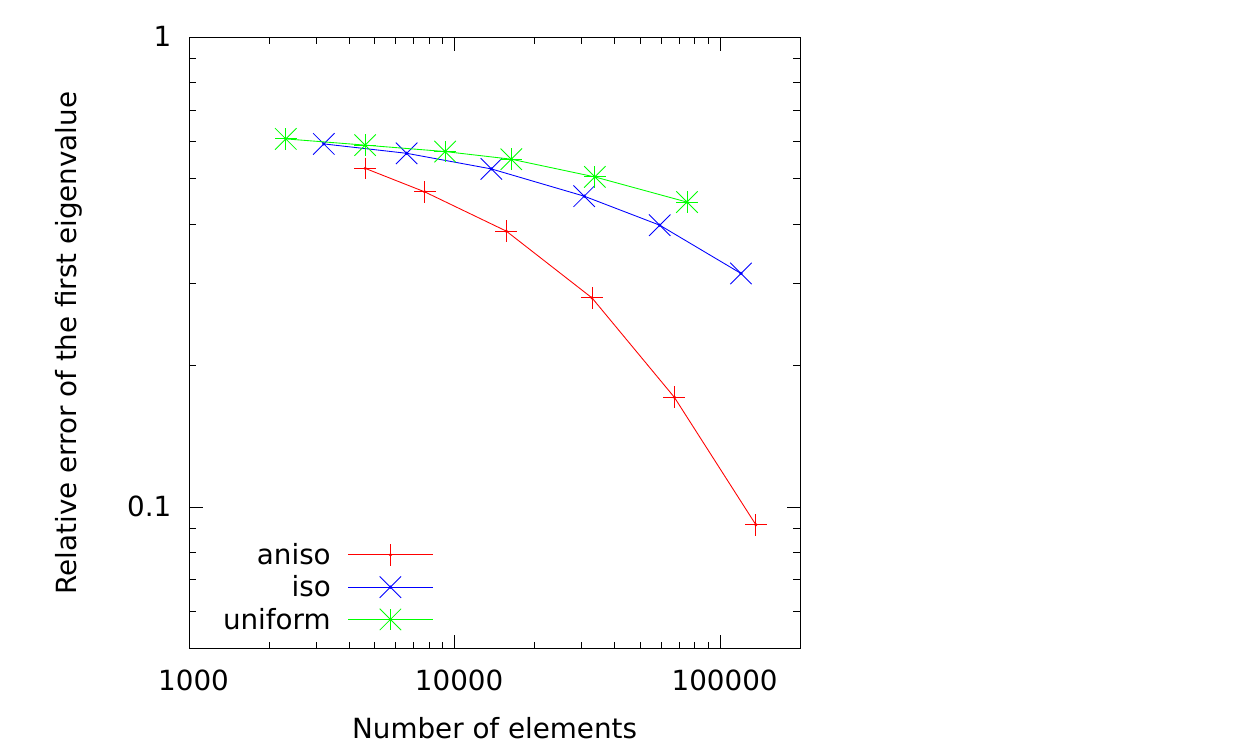}}
\end{minipage}%
\begin{minipage}{.5\linewidth}
\centering
\subfloat[]{\label{error44:b}\includegraphics[scale=.9]{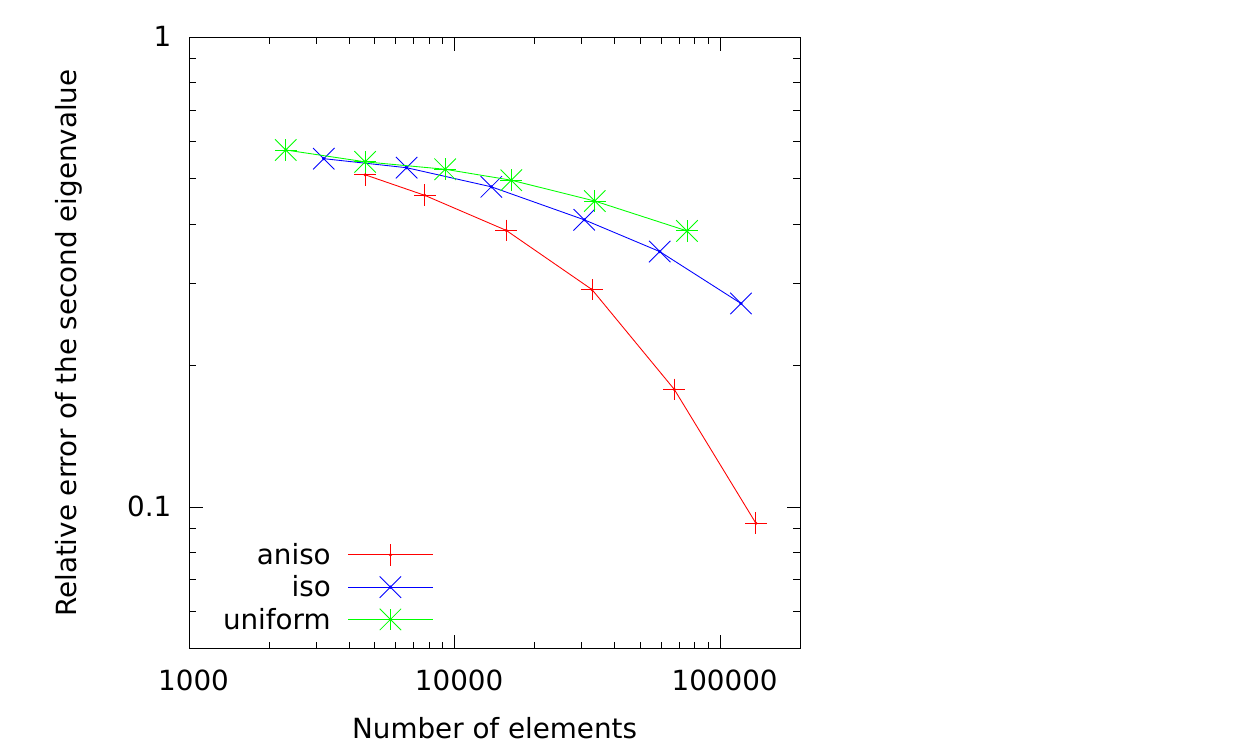}}
\end{minipage}\par\medskip
\begin{minipage}{.5\linewidth}
\centering
\subfloat[]{\label{error44:c}\includegraphics[scale=.9]{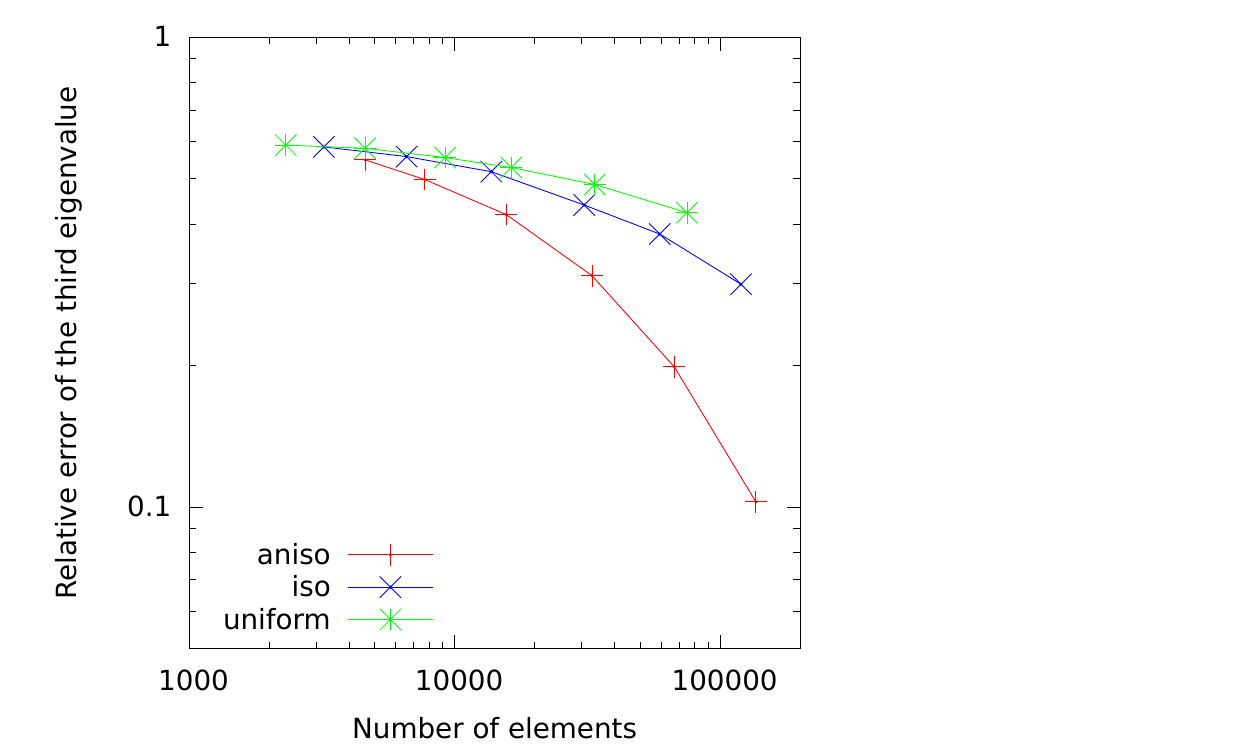}}
\end{minipage}
\begin{minipage}{.5\linewidth}
\centering
\subfloat[]{\label{error44:d}\includegraphics[scale=.9]{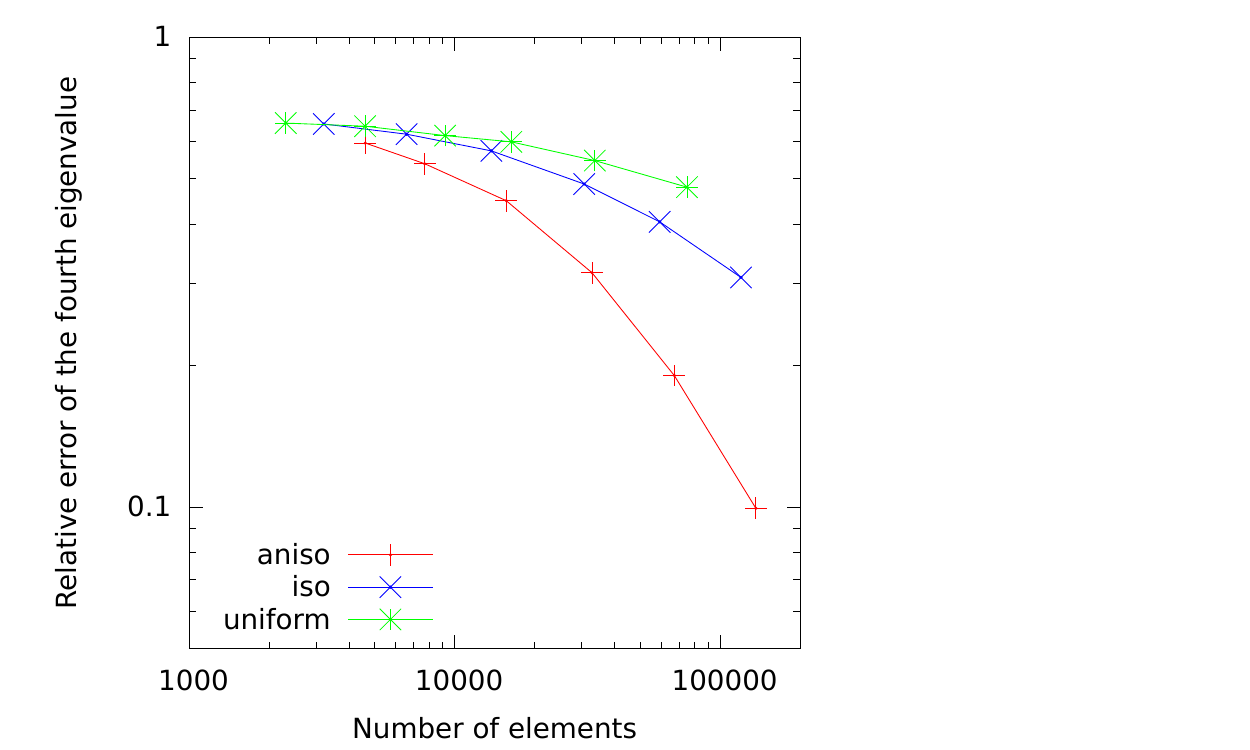}}
\end{minipage}
\caption{The relative error in the first four eigenvalues for the example in \S\ref{exam:bunny}.}
\label{fig:error44}
\end{figure}

\begin{figure}
\begin{minipage}{.5\linewidth}
\centering
\subfloat[]{\label{eigenfun44:a}\includegraphics[scale=.4]{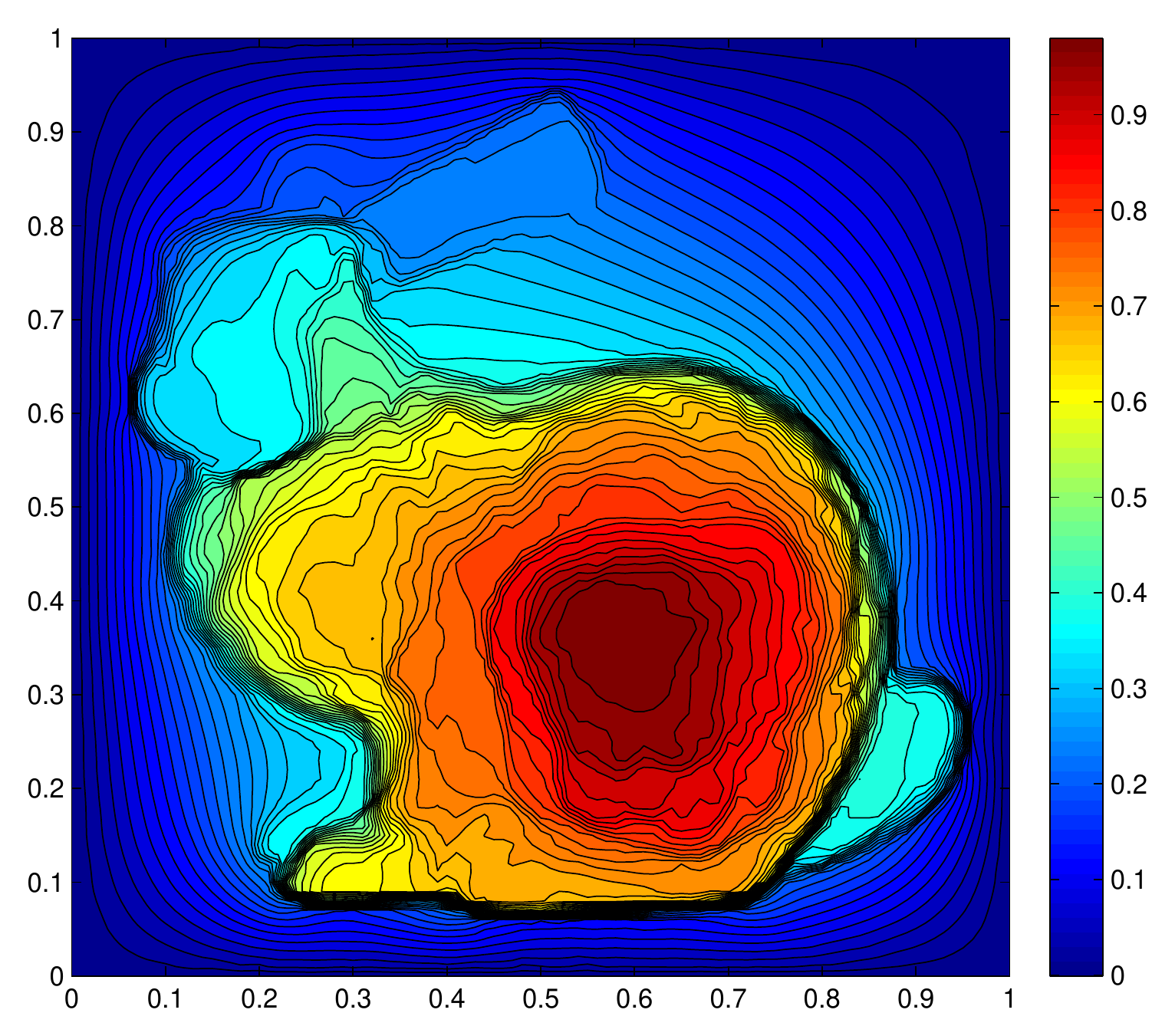}}
\end{minipage}%
\begin{minipage}{.5\linewidth}
\centering
\subfloat[]{\label{eigenfun44:b}\includegraphics[scale=.4]{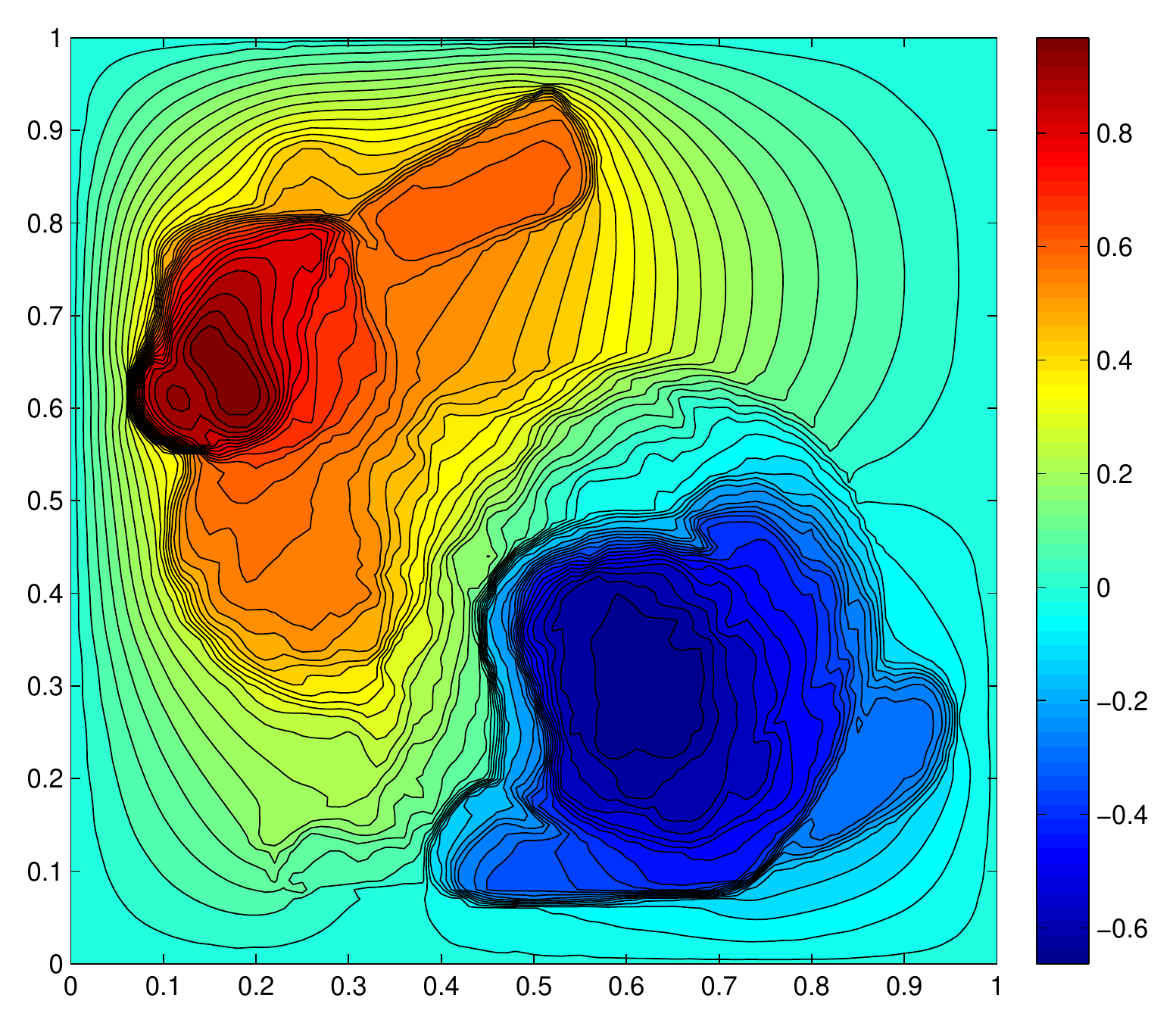}}
\end{minipage}\par\medskip
\begin{minipage}{.5\linewidth}
\centering
\subfloat[]{\label{eigenfun44:c}\includegraphics[scale=.4]{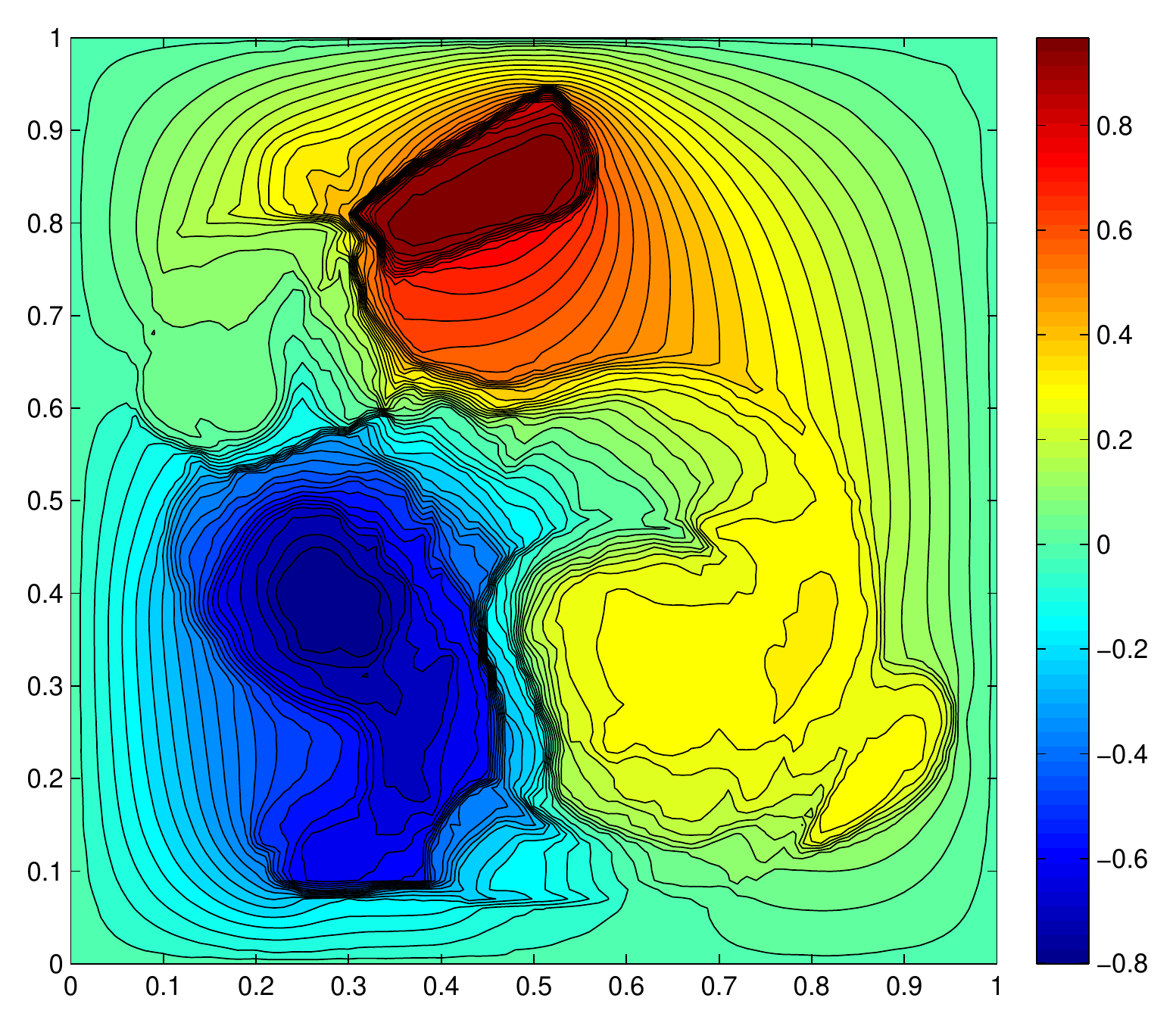}}
\end{minipage}
\begin{minipage}{.5\linewidth}
\centering
\subfloat[]{\label{eigenfun44:d}\includegraphics[scale=.4]{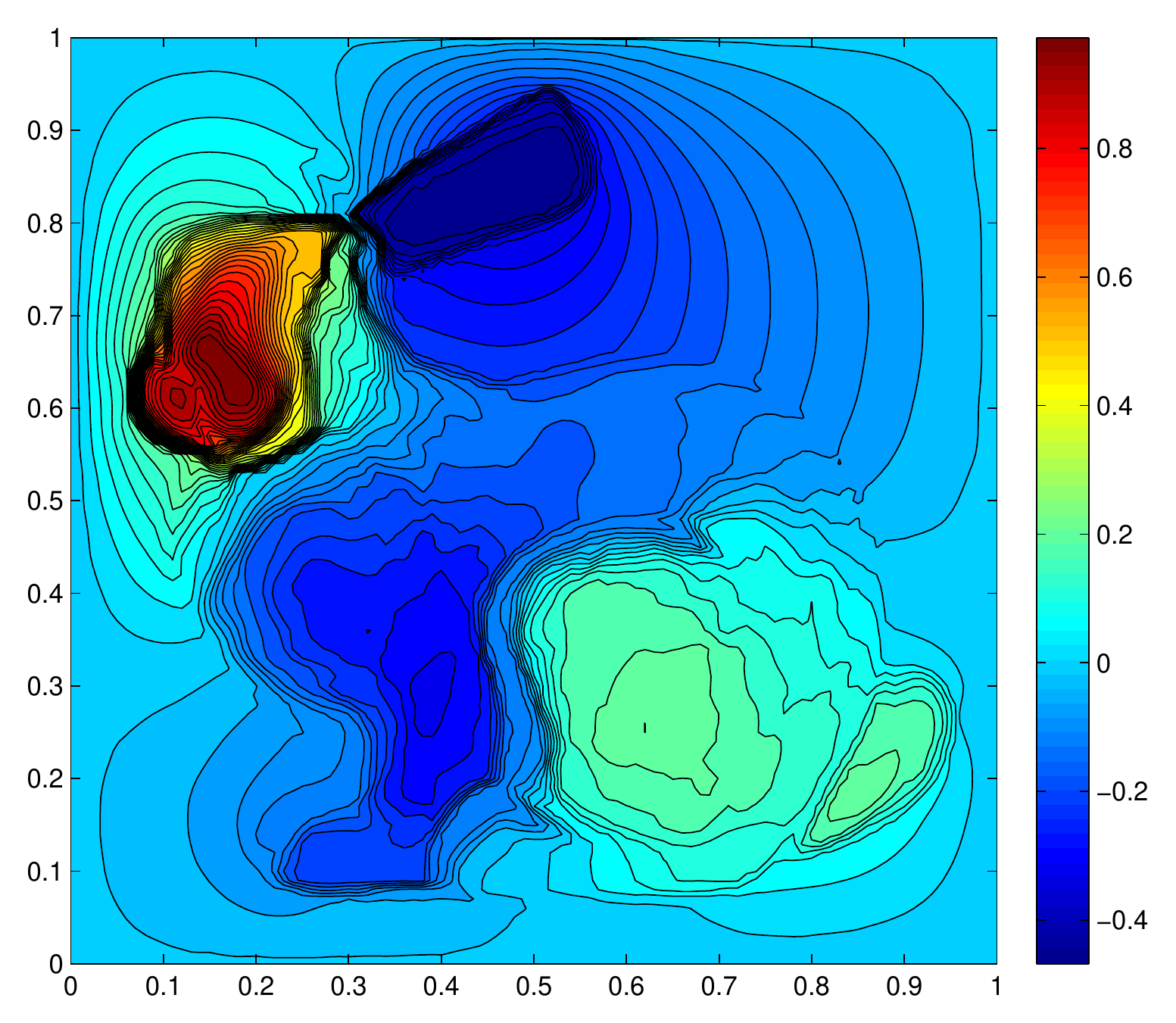}}
\end{minipage}
\caption{The contours of the first four eigenfunctions of the Laplace-Beltrami operator for
the example in \S\ref{exam:bunny}.}
\label{fig:ex44-eigenfun}
\end{figure}

\subsection{Anisotropic eigenvalue problems arising from stability analysis of
nonlinear diffusion equations}
\label{exam:nde}

A number of nonlinear diffusion problems arising in various areas of science and engineering take the form
\begin{equation}
v_t = \nabla \cdot (a(v, \nabla v) \nabla v),
\label{nde-1}
\end{equation}
where $a = a(v, \nabla v)$ is a scalar diffusion coefficient. One example is
the well-known Perona-Malik anisotropic diffusion filter in image processing \cite{PM90} where
the diffusion coefficient is taken as $a(v, \nabla v) = g(|\nabla v|^2)$ with
$g(\cdot)$ being a smooth decreasing function.  A typical choice is
\[
a = \frac{1}{\left (1 + \frac{1}{\gamma^2} \| \nabla v \|_2^2\right )^{\alpha}} ,
\]
where $\alpha$ and $\gamma$ are two positive parameters. 
Another example is the radiation diffusion equation in radiation hydrodynamics \cite{MM84} where
the diffusion coefficient takes the form
\[
a = \frac{1}{\frac{1}{Z^\beta v^{\frac{3}{4}}} + \frac{\| \nabla v\|_2}{v}},
\]
where $Z$ is the atomic number which has different values for different materials and $\beta$ is a parameter.

We are interested in the eigenvalue problem resulting from the linearization of
the nonlinear diffusion equation (\ref{nde-1}) about a solution.
A motivation is the stability consideration of a given solution.
Another motivation is our curiosity to see if the eigenfunctions of the linearized diffusion operator
can be used to detect the features of an image. Denote a solution or an image by $v_0$.
Then, the linearization of (\ref{nde-1}) about $v_0$ reads as
\[
u_t = \nabla \cdot \left ( \left [ a(v_0, \nabla v_0) I
+ \nabla v_0 \frac{\partial a}{\partial \nabla v}(v_0, \nabla v_0) \right ] \nabla u \right )
+ \nabla \cdot \left ( u \frac{\partial a}{\partial v} \nabla v_0\right ) ,
\]
where $u$ denotes the perturbation. The corresponding anisotropic eigenvalue problem is
\begin{equation}
- \nabla \cdot \left ( \left [ a(v_0, \nabla v_0) I
+ \nabla v_0 \frac{\partial a}{\partial \nabla v}(v_0, \nabla v_0) \right ] \nabla u \right )
- \nabla \cdot \left ( u \frac{\partial a}{\partial v} \nabla v_0\right ) = \lambda u .
\label{nde-2}
\end{equation}

We consider the Perona-Malik filter with $\alpha = 0.5$ and $\gamma = 1$. This form of Perona-Malik
filer can also be seen as a regularized version of the total variation flow. In this case,
the existence and convergence of a finite element solution to (\ref{nde-1}) have been proven by
Feng and Prohl \cite{Feng-Prohl-03}.
The eigenvalue problem (\ref{nde-2}) can be cast in the form (\ref{eigen-1}), with the diffusion matrix given by
\begin{equation}
\bD = \frac{1}{(1+\|\nabla v_0 \|_2^2)^{\frac{3}{2}}} \left ( \begin{array}{rr}
1 + (\frac{\partial v_0}{\partial y})^2 & - \frac{\partial v_0}{\partial x} \frac{\partial v_0}{\partial y} \\
- \frac{\partial v_0}{\partial x} \frac{\partial v_0}{\partial y}& 1 + (\frac{\partial v_0}{\partial x})^2 \end{array} \right ),
\quad \rho = 1.
\label{nde-3}
\end{equation}
This problem is similar to the Laplace-Beltrami operator on parametrized surfaces considered in
the previous subsection (cf. (\ref{LB-2})) but differs in $\rho$ and the power of $(1+\|\nabla v_0 \|_2^2)$
(or $(1+\|\nabla \psi \|_2^2$)) in the definition of $\bD$.
In our computation, we consider the initial moment and take $v_0 = \psi(x,y)$, where $\psi(x,y)$ is the surface
defined by the gray level image of the Stanford bunny shown in Fig.~\ref{stanford-bunny-f1:a}.

It is well known that the Perona-Malik model yields enhanced edges while reducing noise
\cite{ALM92,CLMC92,PM90}. The anisotropic eigenvalue problem (\ref{nde-2})
also has this feature; see Figs.~\ref{fig:ex41-eigenfun}, \ref{fig:mesh41}, and \ref{fig:ex41comp}
which show the contours of the first four eigenfunctions, adaptive meshes, and the surface plot and
the re-interpolated image of the first eigenfunction, respectively.
Particularly, we can see in Fig.~\ref{fig:ex41-eigenfun} that the contour lines of the eigenfunctions 
are compressed near the edges and there are fewer contour lines in other regions.
This indicates that the eigenfunctions are close to being piecewise constant (also see Fig.\ref{fig:ex41comp}).
The generated anisotropic meshes also reflect this feature. A close-up 
view near point $(0.6, 0.6)$ in Fig.~\ref{mesh41:d} shows that the mesh elements have a very strong orientation, 
almost ``flat'' and aligned with the edges. The concentration of the mesh elements near the edges
is useful for locating the edges accurately.
For comparison, we also include results obtained with an isotropic adaptive mesh in
Figs.~\ref{fig:mesh41} and \ref{fig:ex41comp}. The advantages of using anisotropic meshes are clear.
The convergence history for the error in eigenvalues is similar to that for the example in \S\ref{exam:bunny}.

\begin{figure}
\begin{minipage}{.5\linewidth}
\centering
\subfloat[]{\label{eigenfun41:a}\includegraphics[scale=.4]{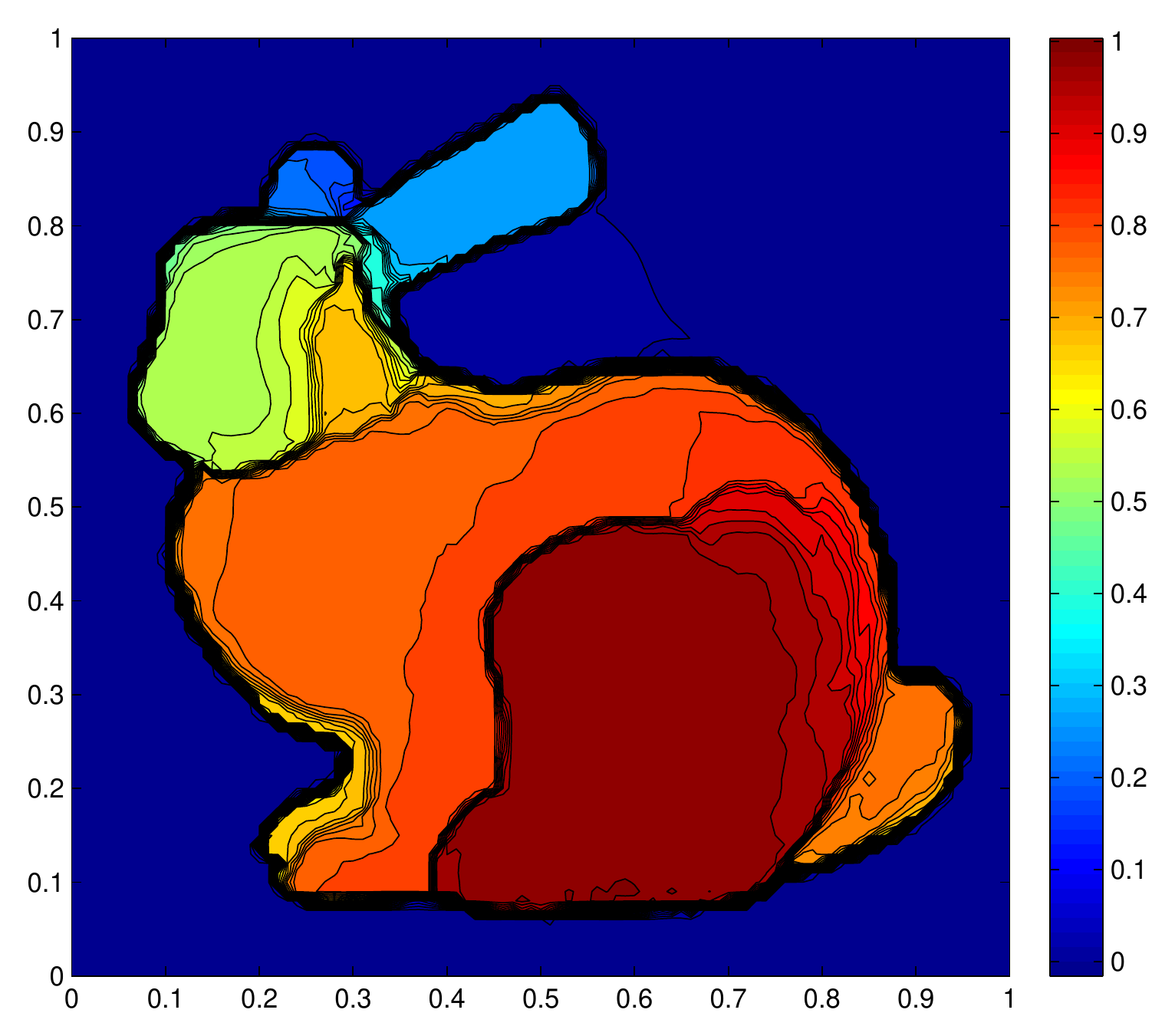}}
\end{minipage}%
\begin{minipage}{.5\linewidth}
\centering
\subfloat[]{\label{eigenfun41:b}\includegraphics[scale=.4]{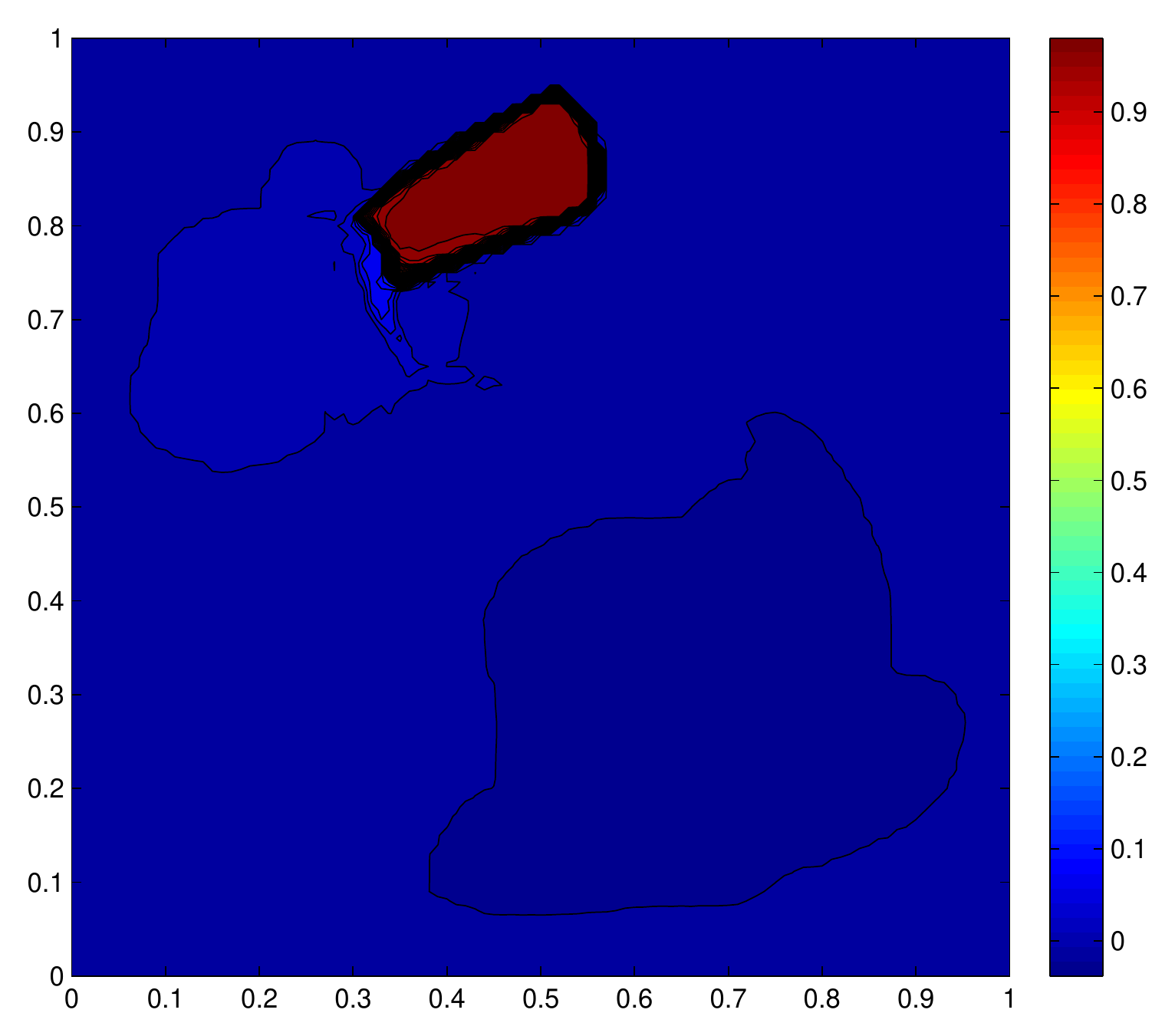}}
\end{minipage}\par\medskip
\begin{minipage}{.5\linewidth}
\centering
\subfloat[]{\label{eigenfun41:c}\includegraphics[scale=.4]{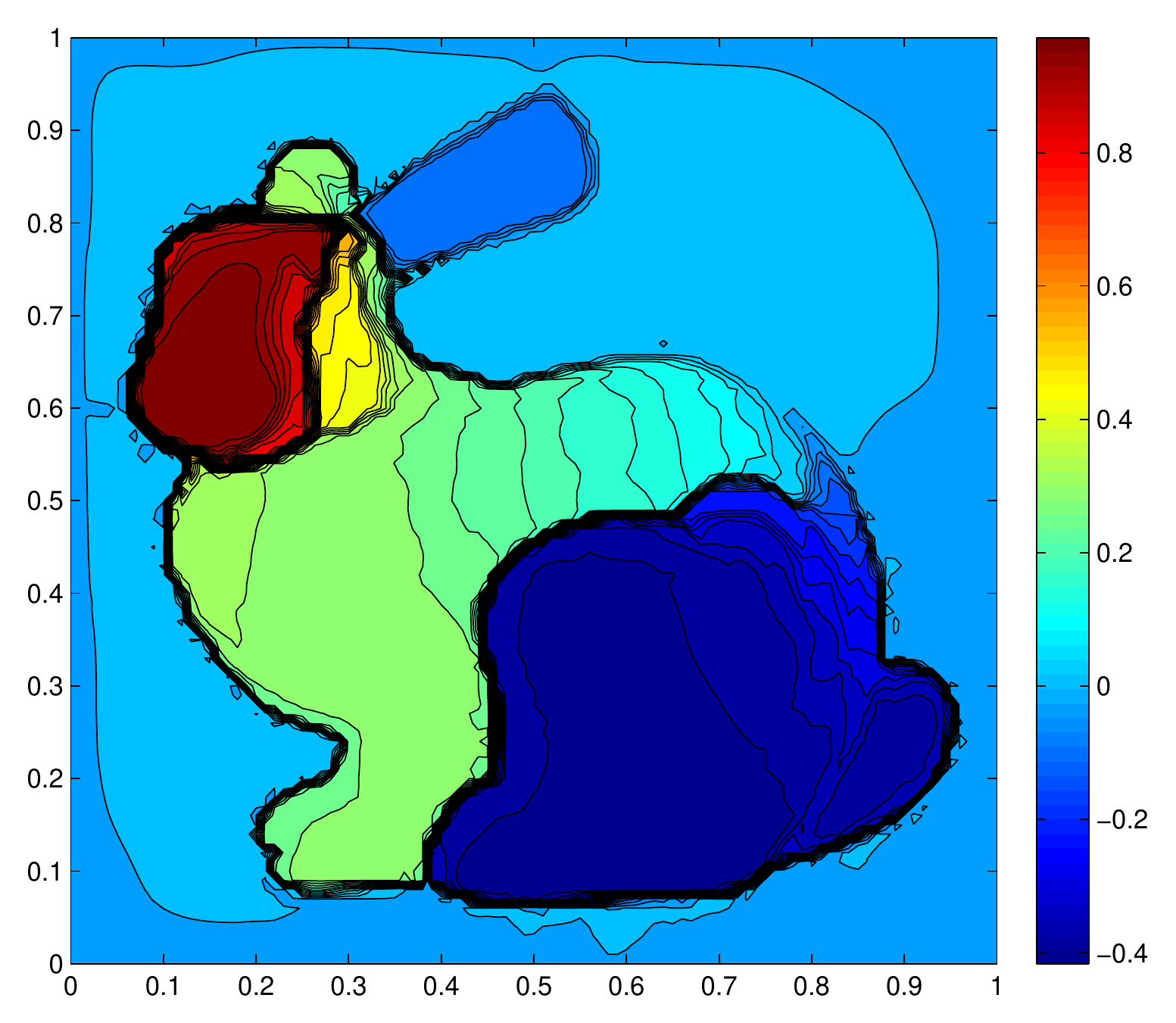}}
\end{minipage}
\begin{minipage}{.5\linewidth}
\centering
\subfloat[]{\label{eigenfun41:d}\includegraphics[scale=.4]{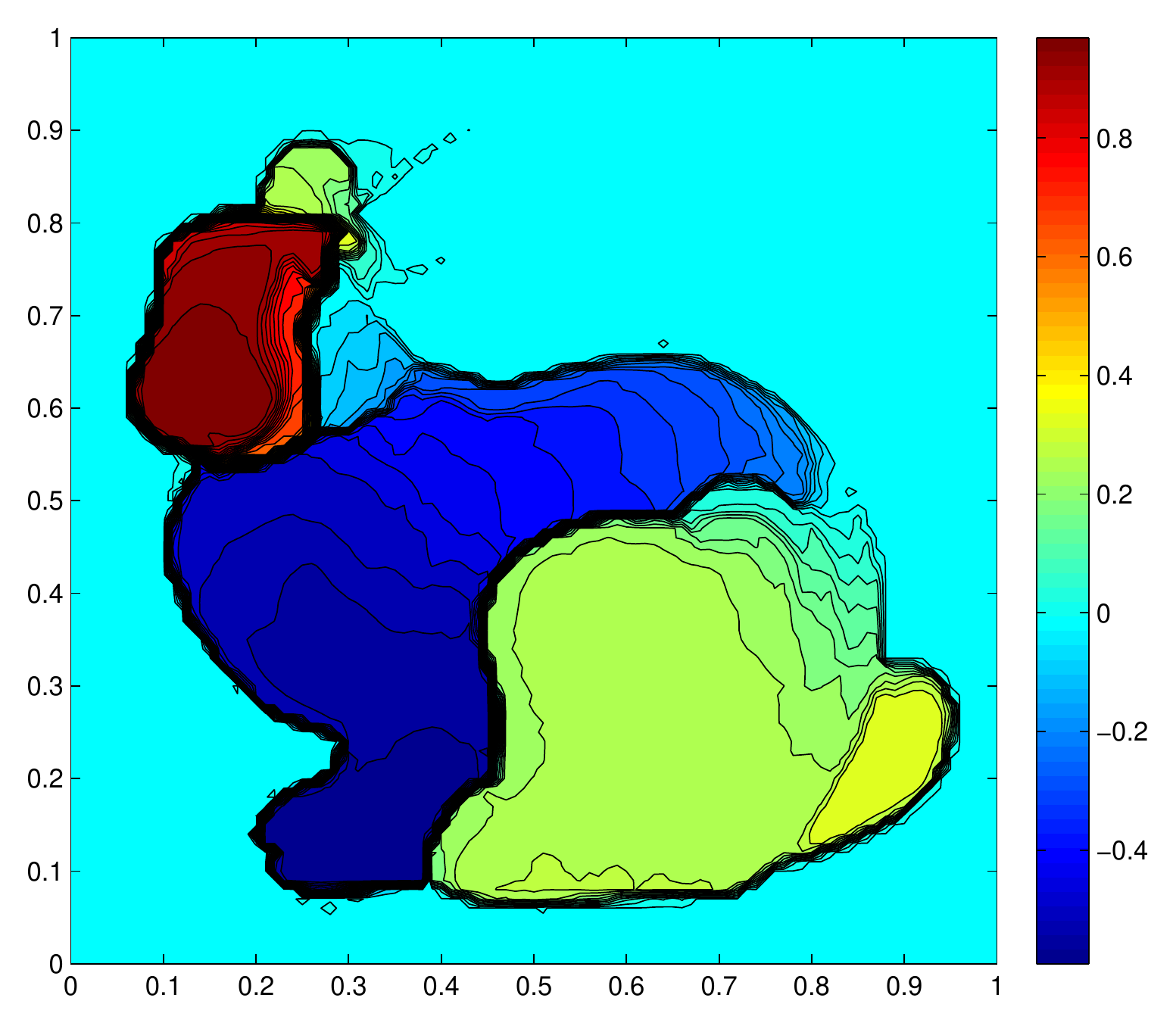}}
\end{minipage}
\caption{The contours of the first four eigenfunctions of the anisotropic eigenvalue problem (\ref{nde-2})
for the example in \S\ref{exam:nde}.}
\label{fig:ex41-eigenfun}
\end{figure}

\begin{figure}
\begin{minipage}{.5\linewidth}
\centering
\subfloat[]{\label{mesh41:a}\includegraphics[scale=.8]{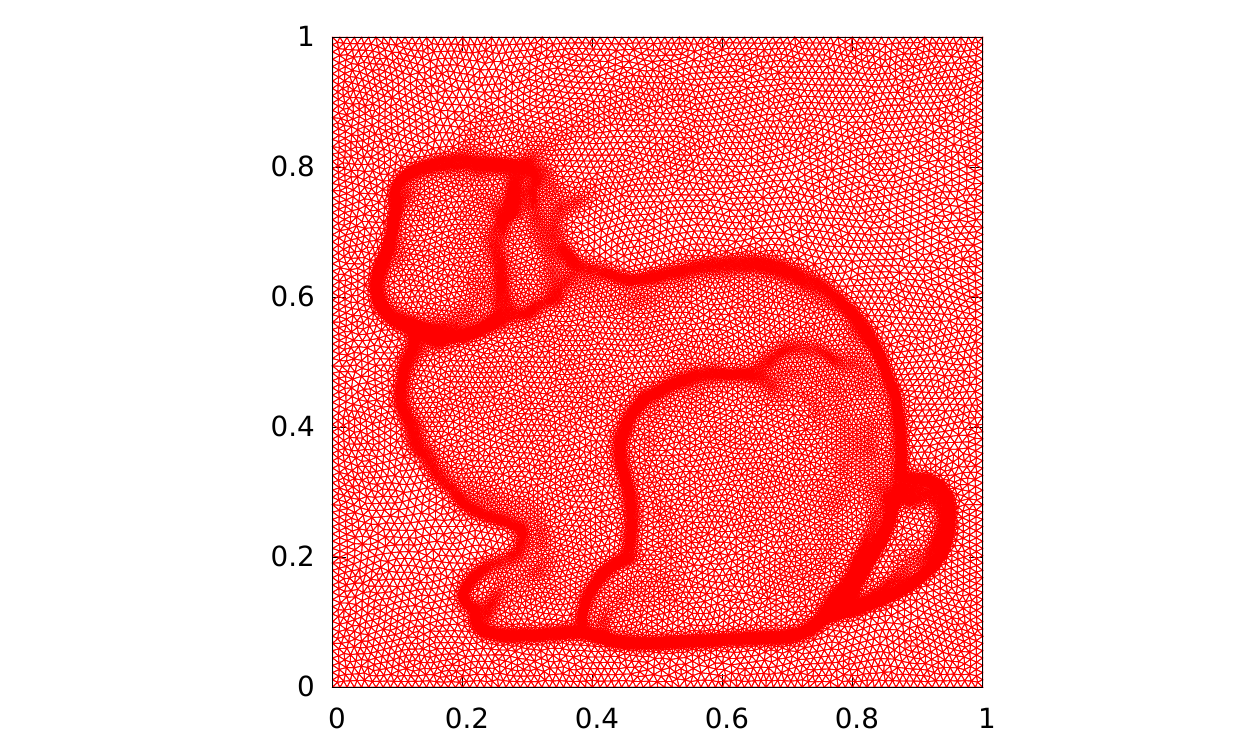}}
\end{minipage}%
\begin{minipage}{.5\linewidth}
\centering
\subfloat[]{\label{mesh41:b}\includegraphics[scale=.8]{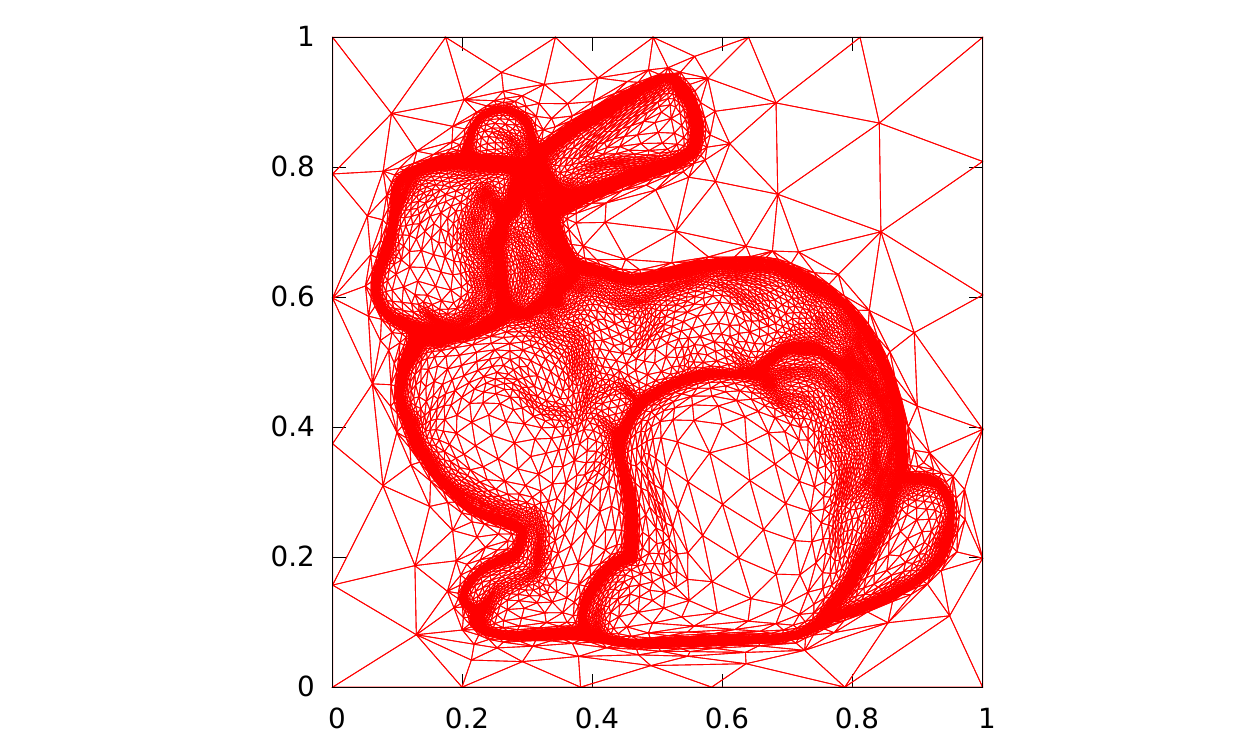}}
\end{minipage}\par\medskip
\begin{minipage}{.5\linewidth}
\centering
\subfloat[]{\label{mesh41:c}\includegraphics[scale=.8]{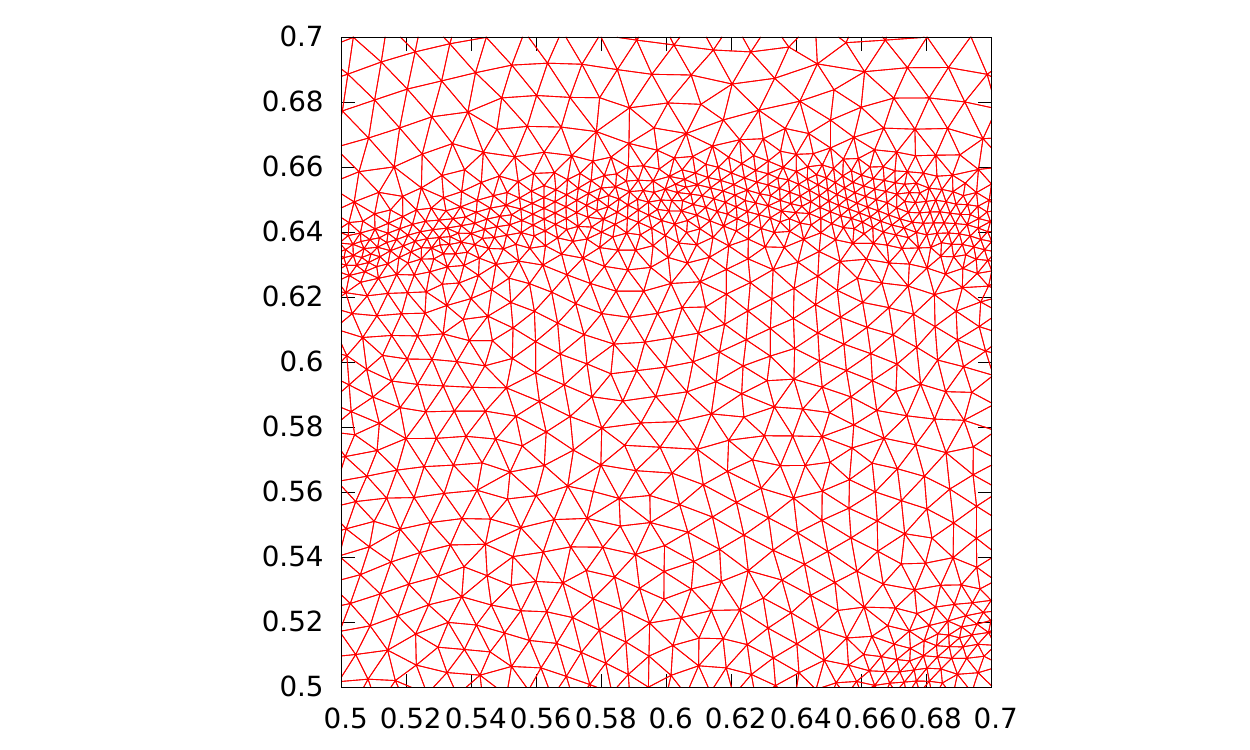}}
\end{minipage}
\begin{minipage}{.5\linewidth}
\centering
\subfloat[]{\label{mesh41:d}\includegraphics[scale=.8]{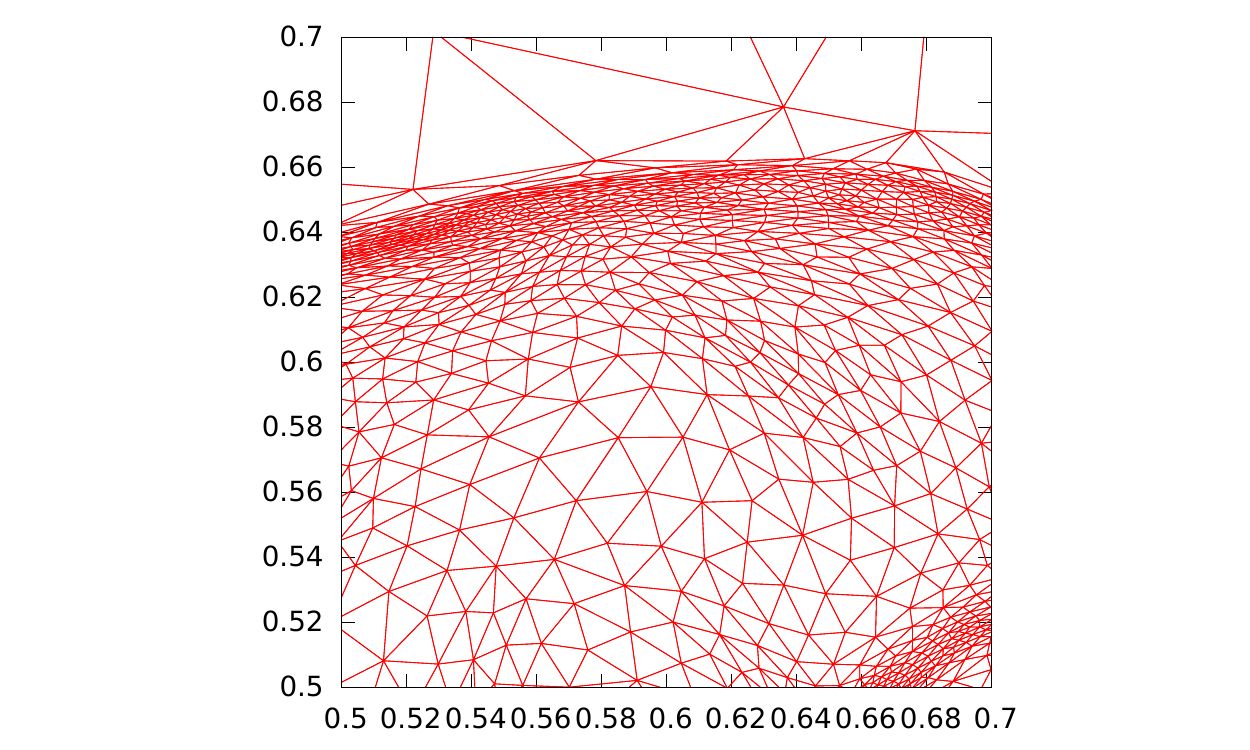}}
\end{minipage}
\caption{Typical adaptive meshes and close-up view near $(0.6, 0.6)$ for the example in \S\ref{exam:nde}.
(a) and (c) are for an isotropic mesh with 52,336 triangular elements and (b) and (d) are for an anisotropic
adaptive mesh with 52,583 triangular elements.}
\label{fig:mesh41}
\end{figure}

\begin{figure}
\begin{minipage}{.5\linewidth}
\centering
\subfloat[]{\label{ex41comp:a}\includegraphics[scale=.4]{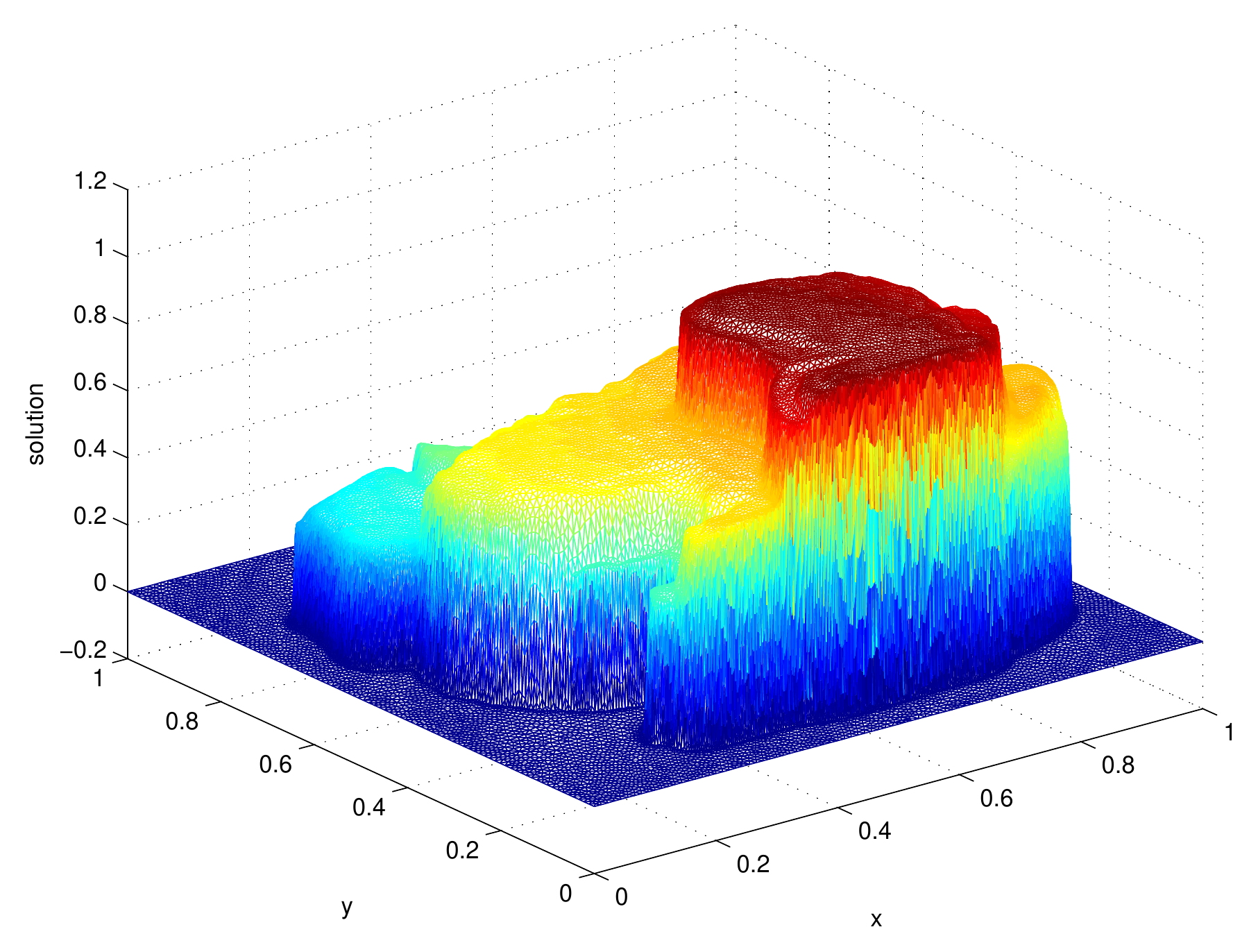}}
\end{minipage}%
\begin{minipage}{.5\linewidth}
\centering
\subfloat[]{\label{ex41comp:b}\includegraphics[scale=.4]{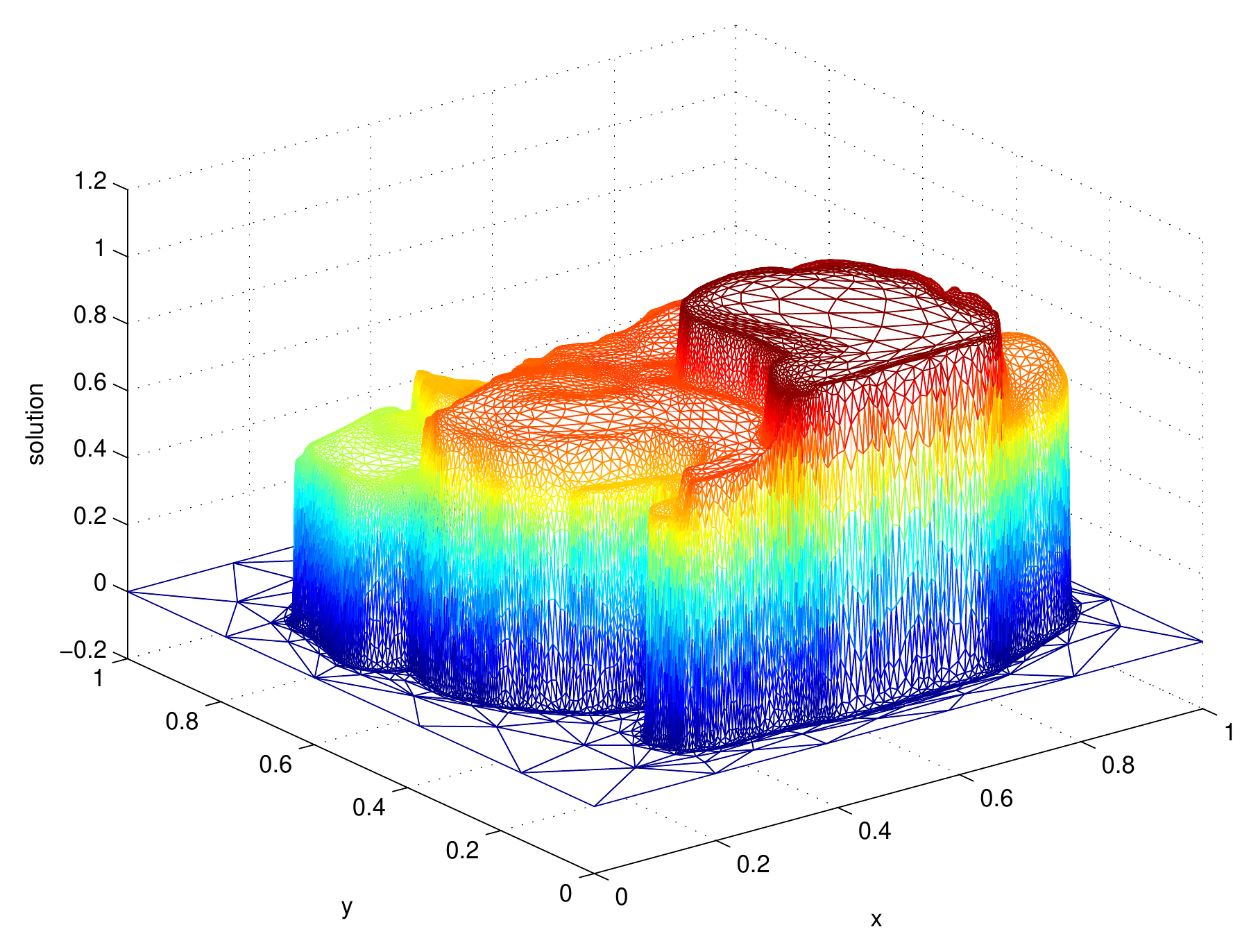}}
\end{minipage}\par\medskip
\begin{minipage}{.5\linewidth}
\centering
\subfloat[]{\label{ex41comp:c}\includegraphics[scale=.6]{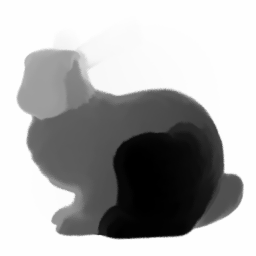}}
\end{minipage}%
\begin{minipage}{.5\linewidth}
\centering
\subfloat[]{\label{ex41comp:d}\includegraphics[scale=.6]{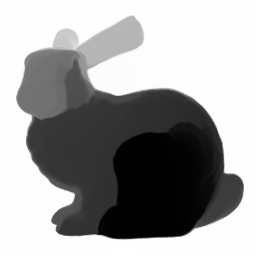}}
\end{minipage}
\caption{For the example in \S\ref{exam:nde}, the surface of the first eigenfunction and the corresponding re-interpolated
image are plotted. (a) and (c) are for the isotropic adaptive mesh shown in Fig.~\ref{mesh41:a} and
(b) and (d) are for the anisotropic adaptive mesh shown in Fig.~\ref{mesh41:b}.}
\label{fig:ex41comp}
\end{figure}

\subsection{Thermal diffusion in a magnetic field in plasma physics}
\label{exam:ring-test}

This example is motivated from the ring diffusion test (e.g., see Parrish and Stone \cite{PS05}
and Sharma and Hammett \cite{SH07}) for thermal diffusion in a magnetic field in plasma physics.
The problem is in the form (\ref{eigen-1})  with $\rho = 1$ and the diffusion matrix
\[
\mathbb{D}=\chi_{\parallel}\V{b}\V{b}^{t}+\chi_{\perp}(I-\V{b}\V{b}^{t}),
\]
where $\V{b}=\V{e}_{z}\times\nabla\psi/\| \V{e}_{z}\times\nabla\psi\|_2$ is the unit direction along the magnetic field line,
$\psi$ is the magnetic potential, and $\chi_{\parallel}$ and $\chi_{\perp}$ are the coefficients
of parallel and perpendicular conduction with respect to $\V{b}$.
In our computation, we choose $\chi_{\parallel}=10^{3}$, $\chi_{\perp}=1$, and a circular magnetic field as
\[
\psi=\sqrt{x^{2}+ y^{2}},\quad (x, y) \in \Omega = (-1, 1)\times (-1, 1) .
\]

A typical anisotropic adaptive mesh for this example is shown in Fig.~\ref{fig:mesh30}. One can see
that most elements are isotropic except in the regions near the unit circle. 
The relative error in the first four eigenvalues is plotted in Fig.~\ref{fig:error30} as a function
of the number of elements. The results show that although this problem does not have
very strong anisotropic features, the improvements in accuracy using anisotropic meshes over
isotropic or uniform ones are still significant.
We also plot the contour lines of the first four eigenfunctions in Fig.~\ref{ex4.3-cont}.

\begin{figure}
\centering
\includegraphics[scale=.9]{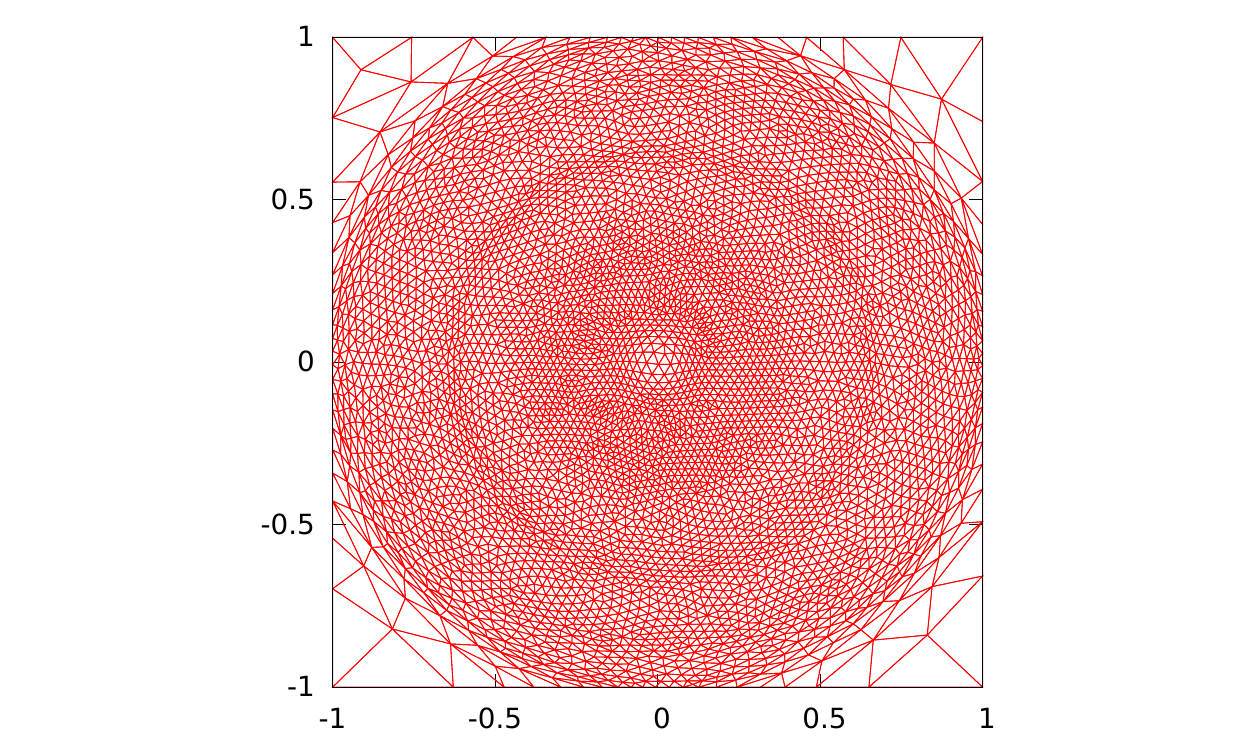}
\caption{A typical adaptive mesh obtained for the example in \S\ref{exam:ring-test}.}
\label{fig:mesh30}
\end{figure}

\begin{figure}
\begin{minipage}{.5\linewidth}
\centering
\subfloat[]{\label{error30:a}\includegraphics[scale=.9]{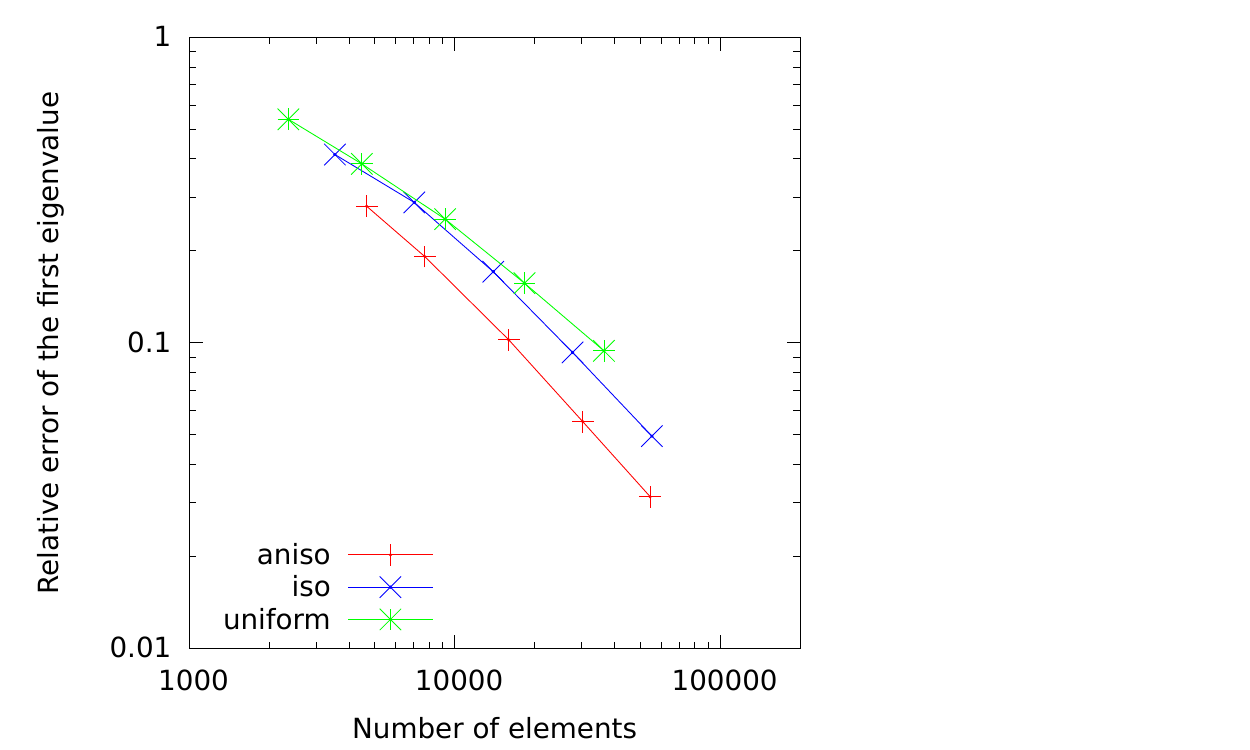}}
\end{minipage}%
\begin{minipage}{.5\linewidth}
\centering
\subfloat[]{\label{error30:b}\includegraphics[scale=.9]{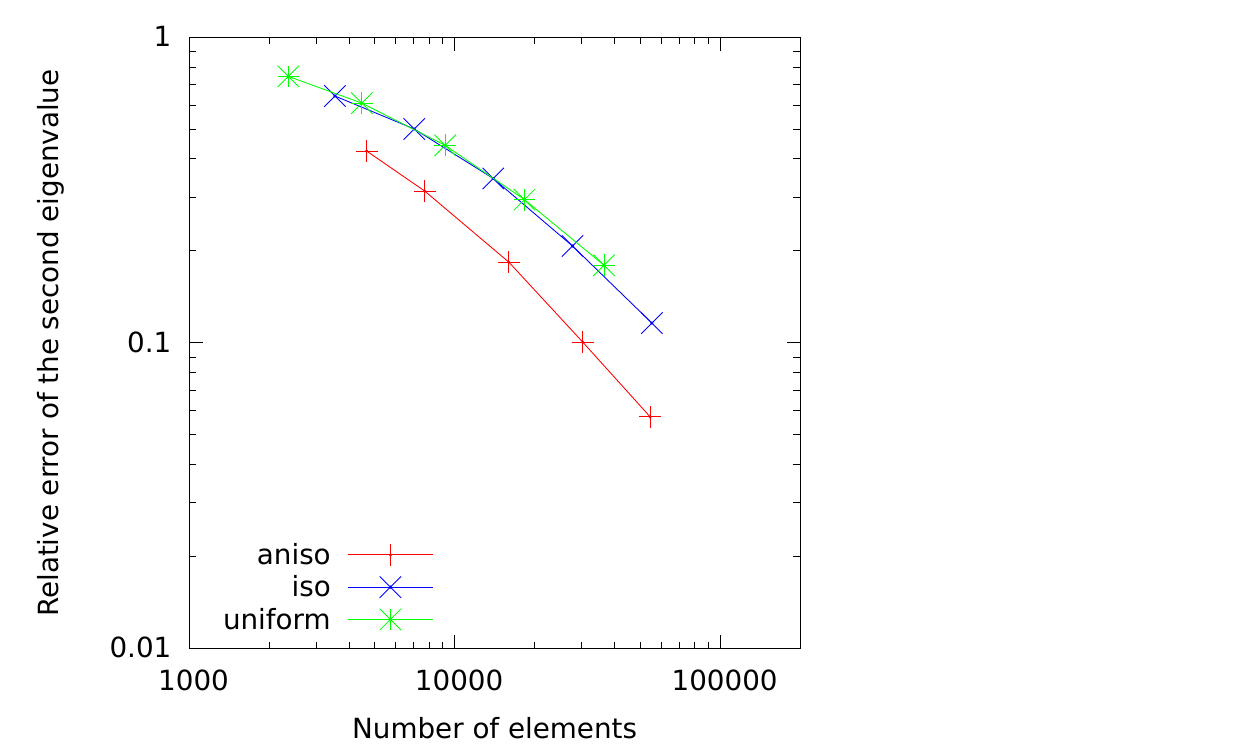}}
\end{minipage}\par\medskip
\begin{minipage}{.5\linewidth}
\centering
\subfloat[]{\label{error30:c}\includegraphics[scale=.9]{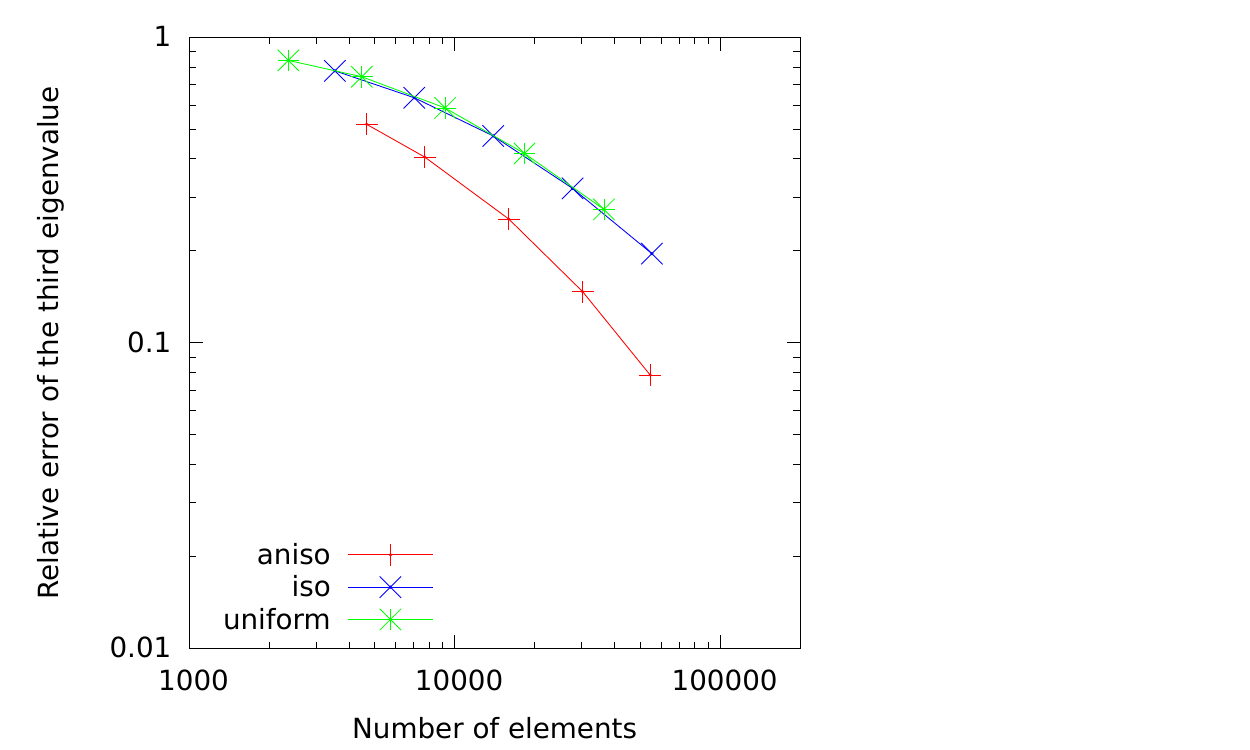}}
\end{minipage}
\begin{minipage}{.5\linewidth}
\centering
\subfloat[]{\label{error30:d}\includegraphics[scale=.9]{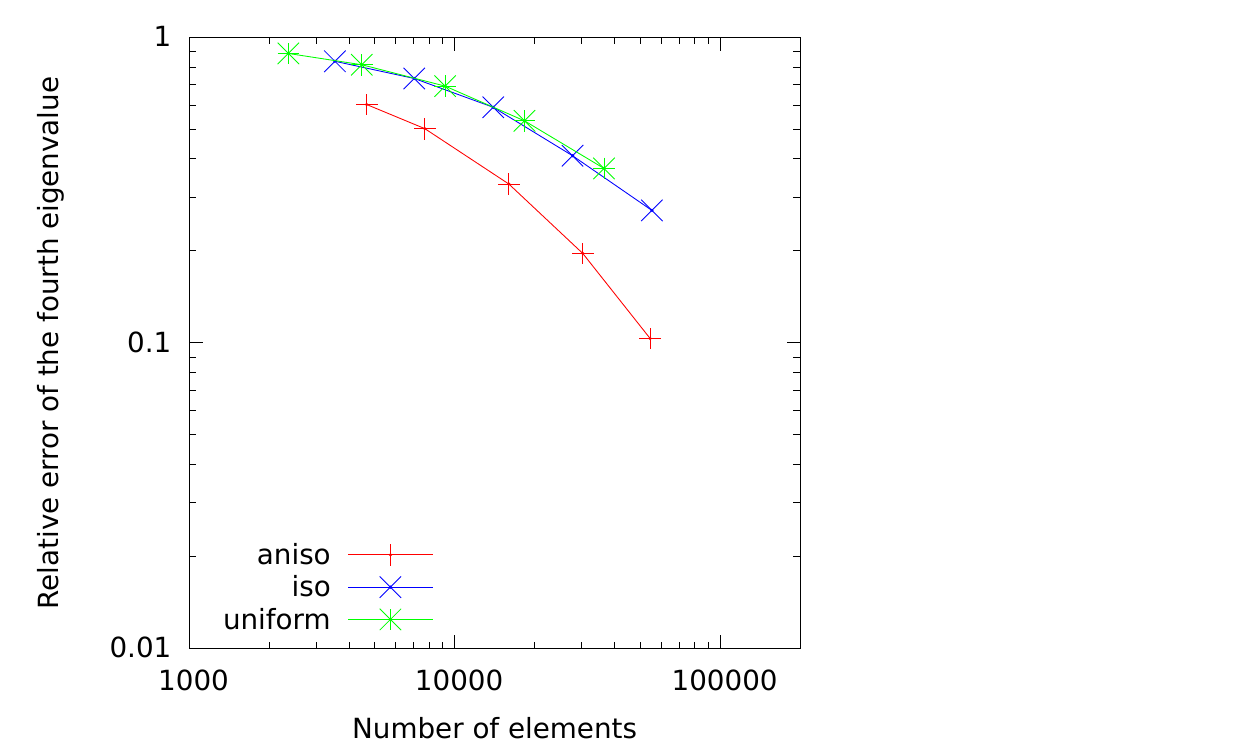}}
\end{minipage}
\caption{For the example in \S\ref{exam:ring-test}, the relative error in the first four eigenvalues is plotted
as a function of the number of mesh elements.}
\label{fig:error30}
\end{figure}

\begin{figure}
\begin{minipage}{.5\linewidth}
\centering
\subfloat[]{\label{cont43:a}\includegraphics[scale=.4]{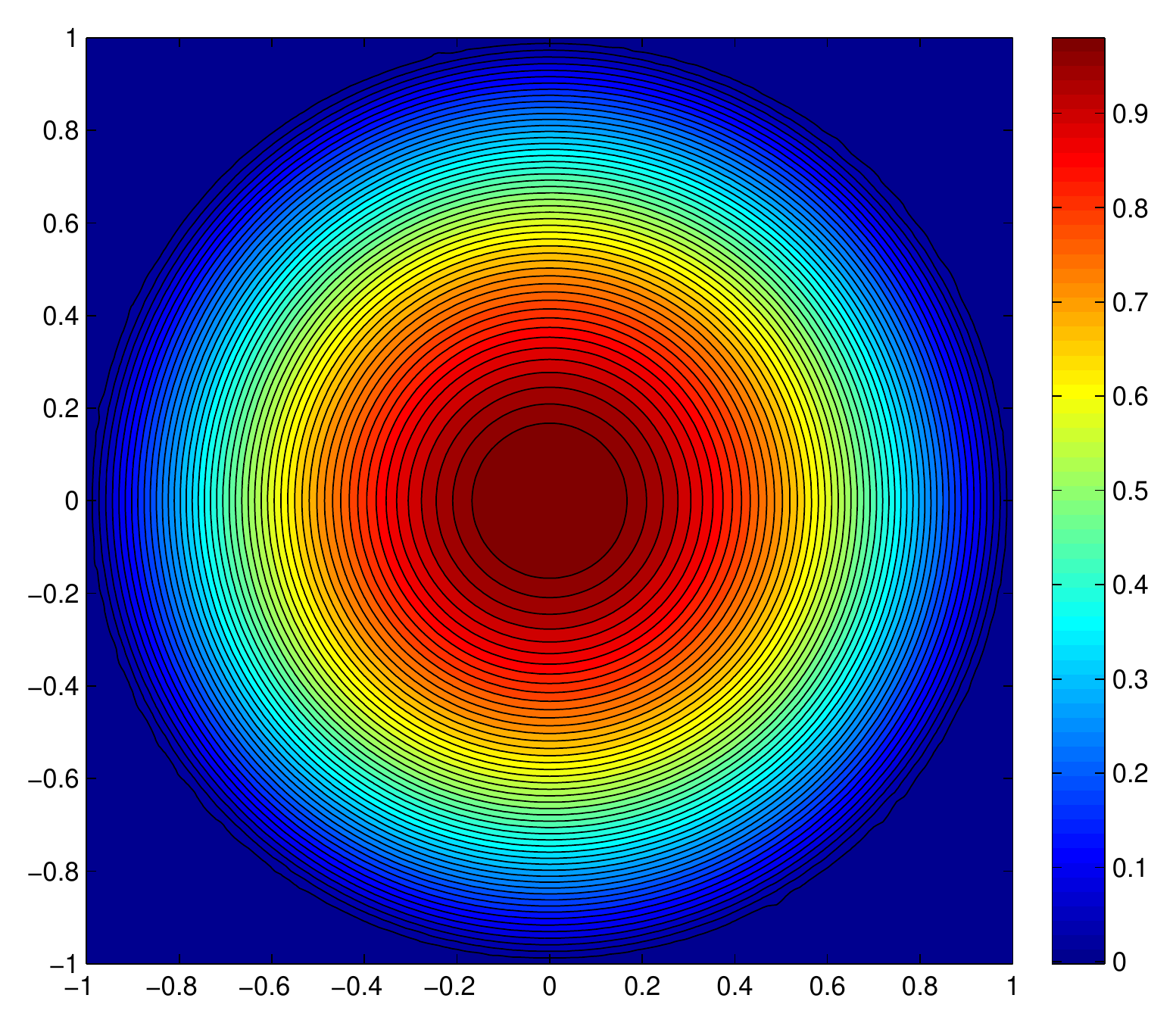}}
\end{minipage}%
\begin{minipage}{.5\linewidth}
\centering
\subfloat[]{\label{cont43:b}\includegraphics[scale=.4]{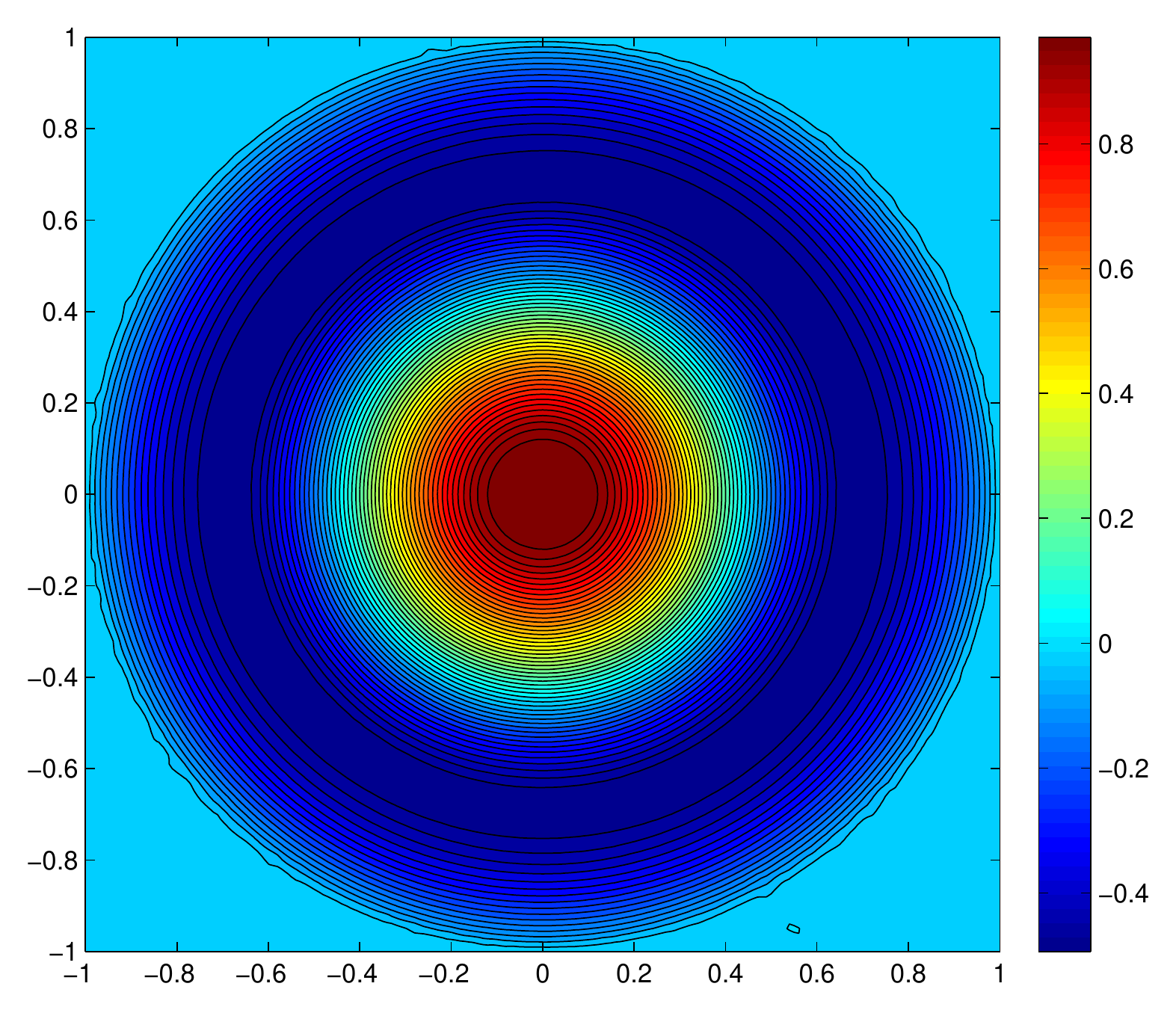}}
\end{minipage}\par\medskip
\begin{minipage}{.5\linewidth}
\centering
\subfloat[]{\label{cont43:c}\includegraphics[scale=.4]{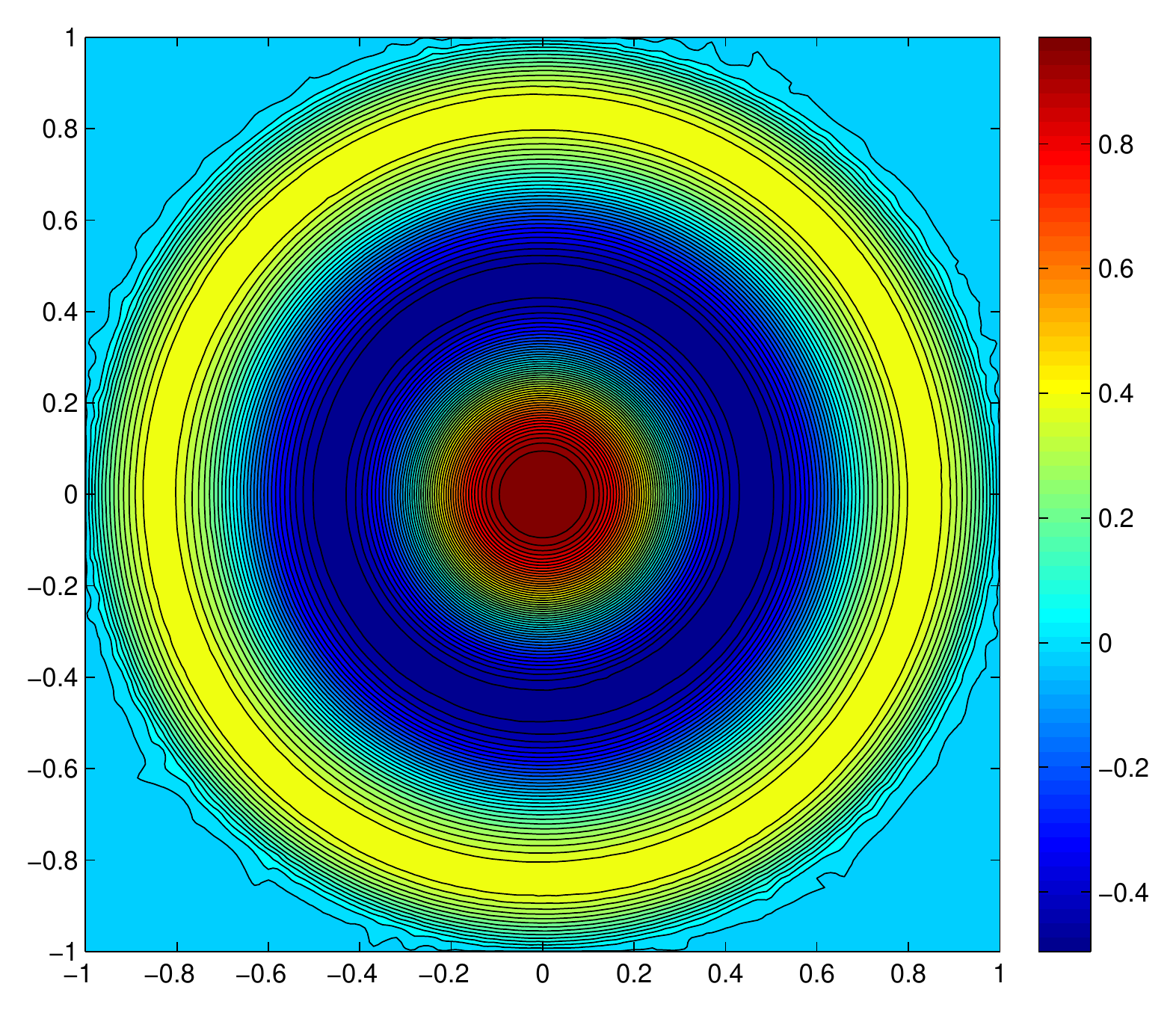}}
\end{minipage}
\begin{minipage}{.5\linewidth}
\centering
\subfloat[]{\label{cont43:d}\includegraphics[scale=.4]{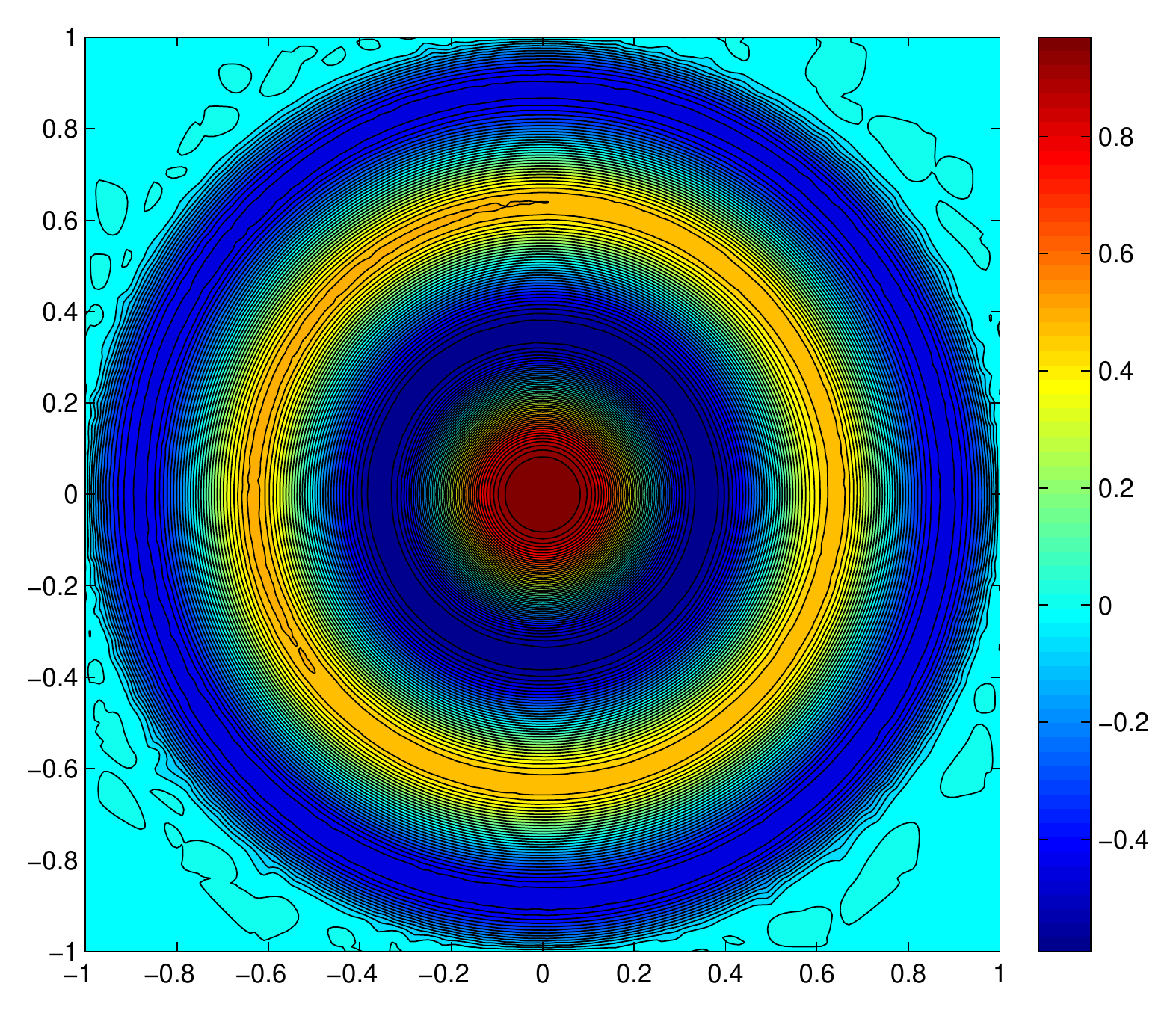}}
\end{minipage}
\caption{For the example in \S\ref{exam:ring-test}, contour lines of the first four eigenfunctions.}
\label{ex4.3-cont}
\end{figure}

\subsection{An eigenvalue problem on an $L$-shape domain}
\label{exam:L-shape}

The last example is in the form of (\ref{eigen-1}) with $\bD = I$, $\rho = 1$, and 
$\Omega$ being an L-shaped domain. This example does not exhibit strong anisotropic behavior
and there is no need for anisotropic mesh adaptation. This example is selected
because the first eigenfunction has singularity at the the re-entrant corner which requires
mesh adaptation and this is a good example to demonstrate that
the metric tensor (\ref{M-1}) will lead to isotropic adaptive meshes for problems
having only isotropic features.

The numerical results are shown in Figs.~\ref{fig:error08} and \ref{fig:mesh08}.
We can see that the mesh generated with (\ref{M-1}) is isotropic, with elements concentrated
near the re-entrant corner. The convergence history for the error in the second to fourth
eigenvalues is almost the same for uniform, isotropic (generated with (\ref{M-3})),
and anisotropic (generated with (\ref{M-1})) meshes.
However, for the first eigenvalue, for which the corresponding eigenfunction
has singularity near the re-entrant corner, the error converges at a slower rate
with uniform meshes than with adaptive meshes.

\begin{figure}
\begin{minipage}{.5\linewidth}
\centering
\subfloat[]{\label{error08:a}\includegraphics[scale=.9]{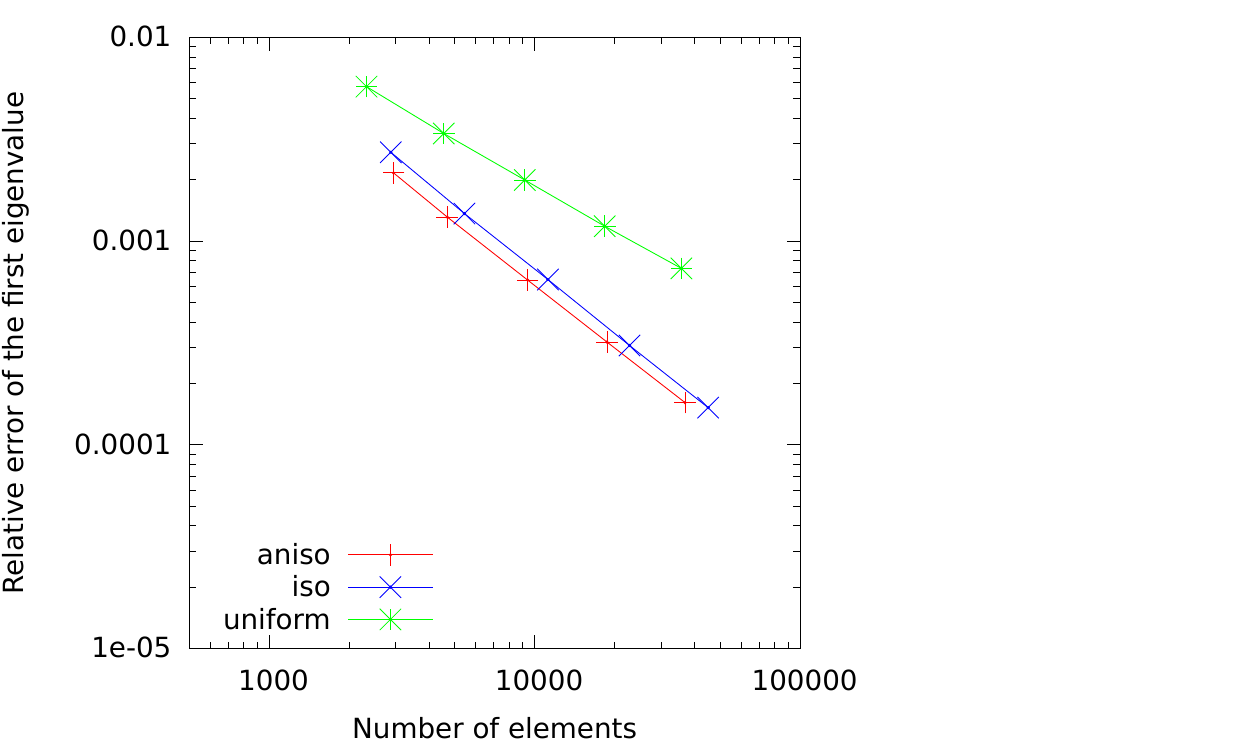}}
\end{minipage}%
\begin{minipage}{.5\linewidth}
\centering
\subfloat[]{\label{error08:b}\includegraphics[scale=.9]{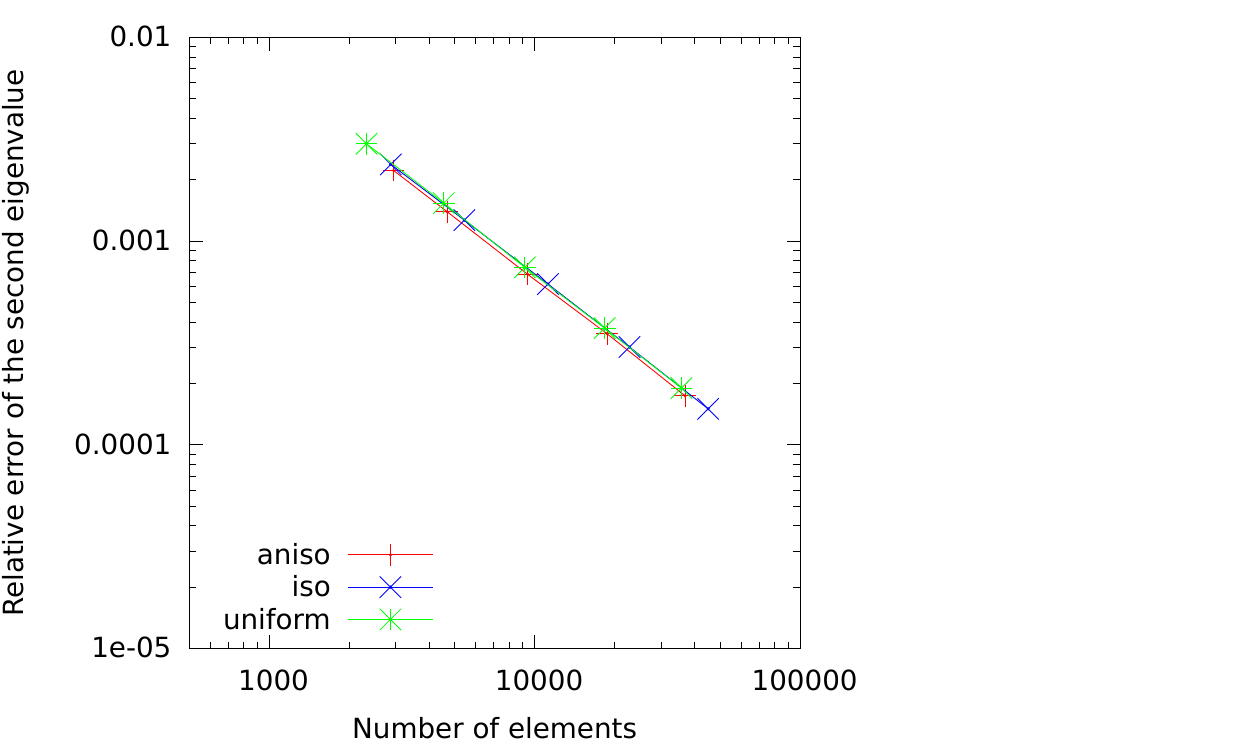}}
\end{minipage}\par\medskip
\begin{minipage}{.5\linewidth}
\centering
\subfloat[]{\label{error08:c}\includegraphics[scale=.9]{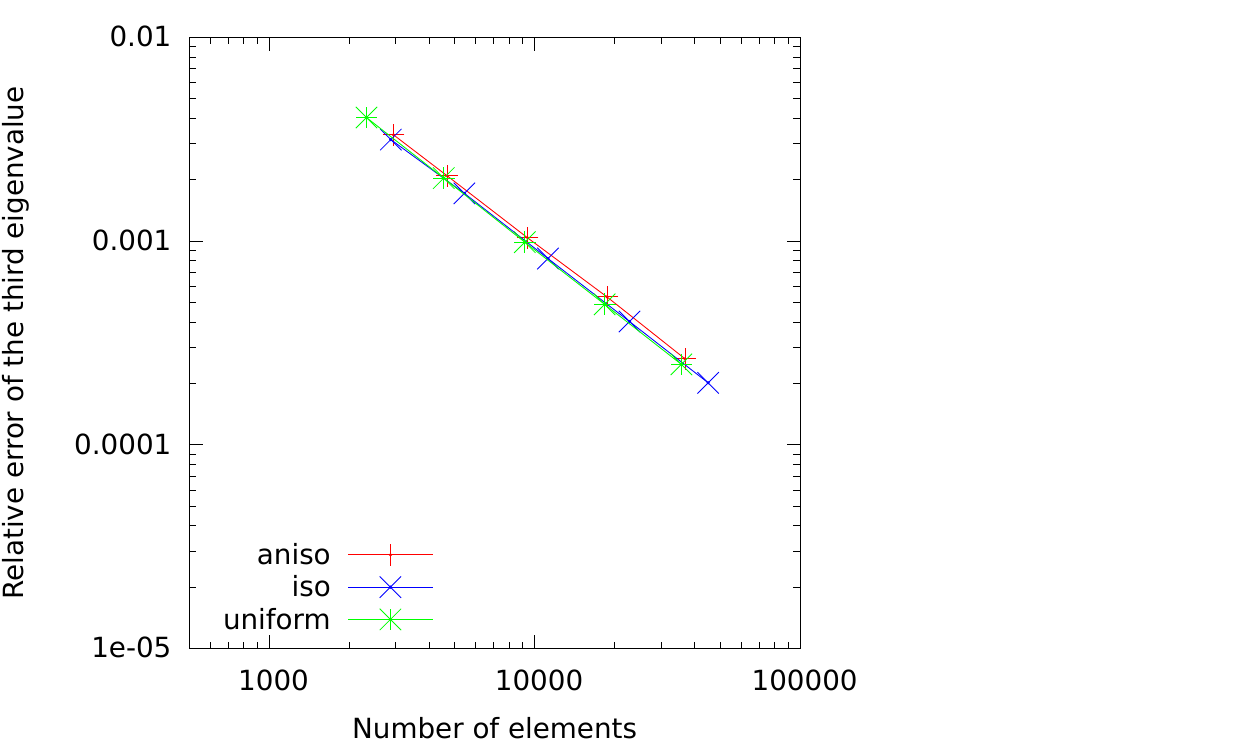}}
\end{minipage}
\begin{minipage}{.5\linewidth}
\centering
\subfloat[]{\label{error08:d}\includegraphics[scale=.9]{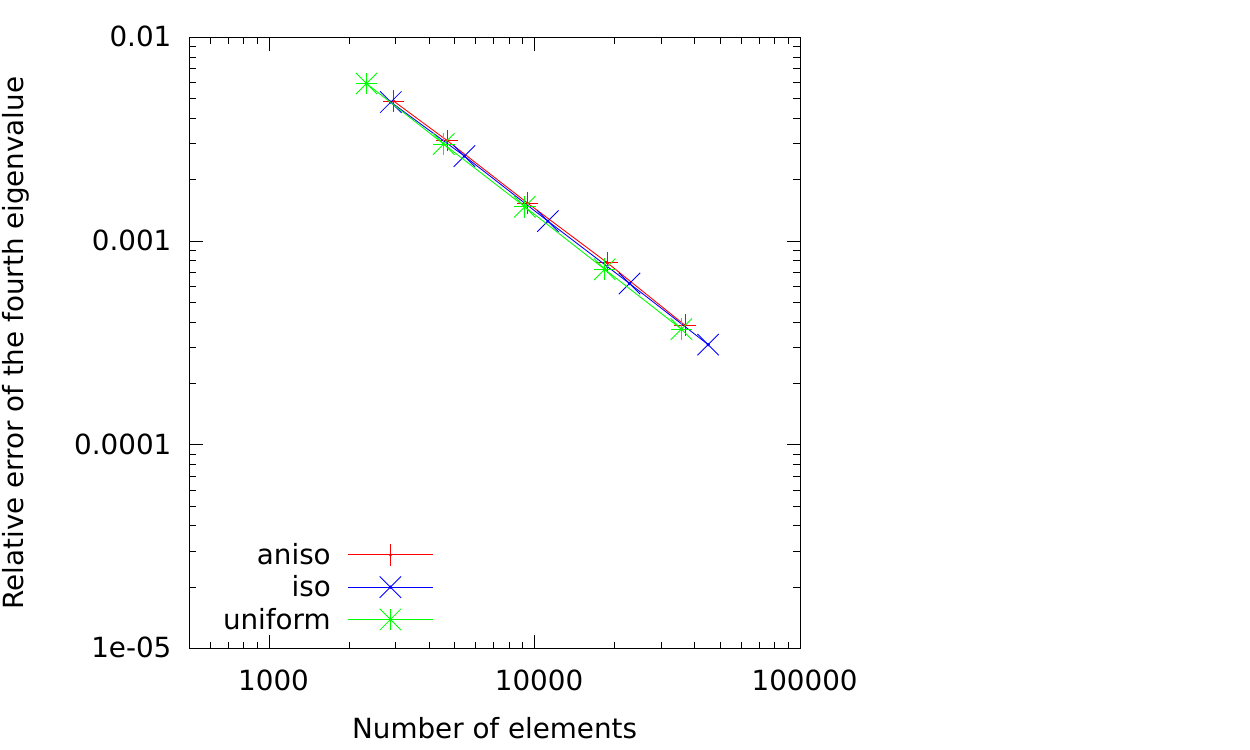}}
\end{minipage}
\caption{For the example in \S\ref{exam:L-shape}, the relative error in the first four eigenvalues
is plotted as a function of the number of mesh elements.}
\label{fig:error08}
\end{figure}

\begin{figure}
\begin{minipage}{.5\linewidth}
\centering
\subfloat[]{\label{mesh08:a}\includegraphics[scale=.9]{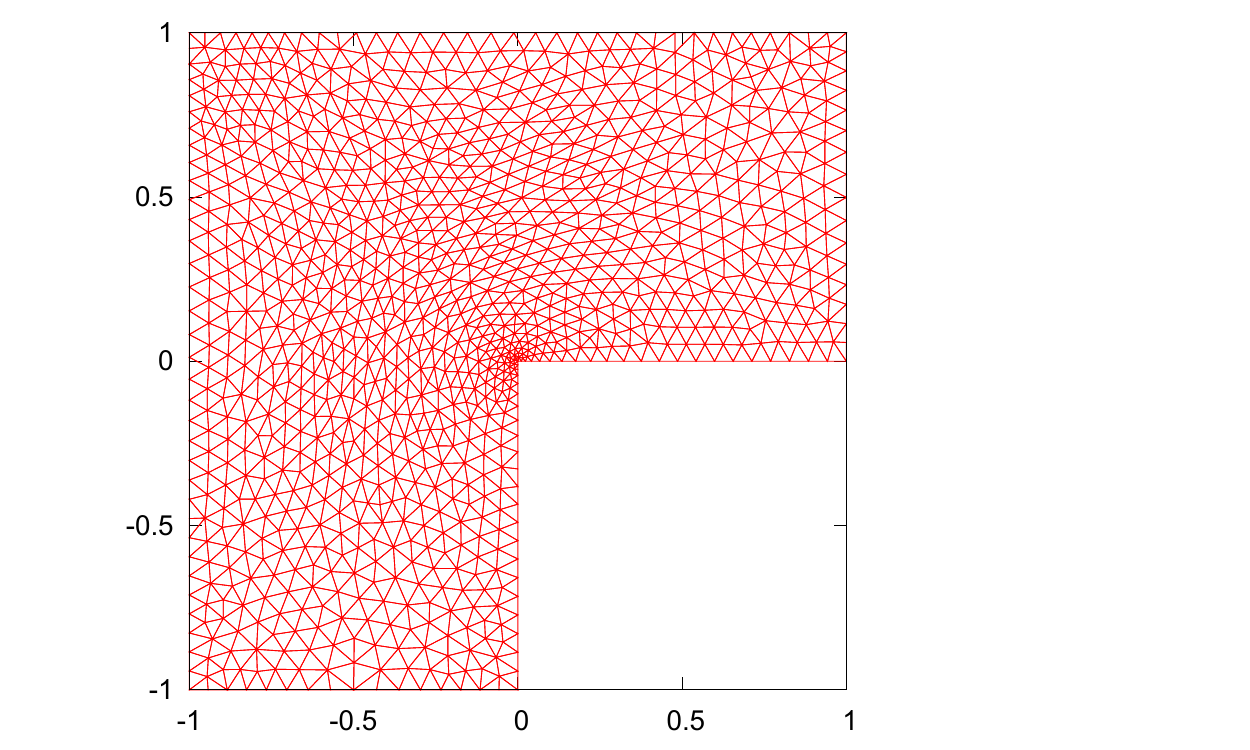}}
\end{minipage}%
\begin{minipage}{.5\linewidth}
\centering
\subfloat[]{\label{mesh08:b}\includegraphics[scale=.9]{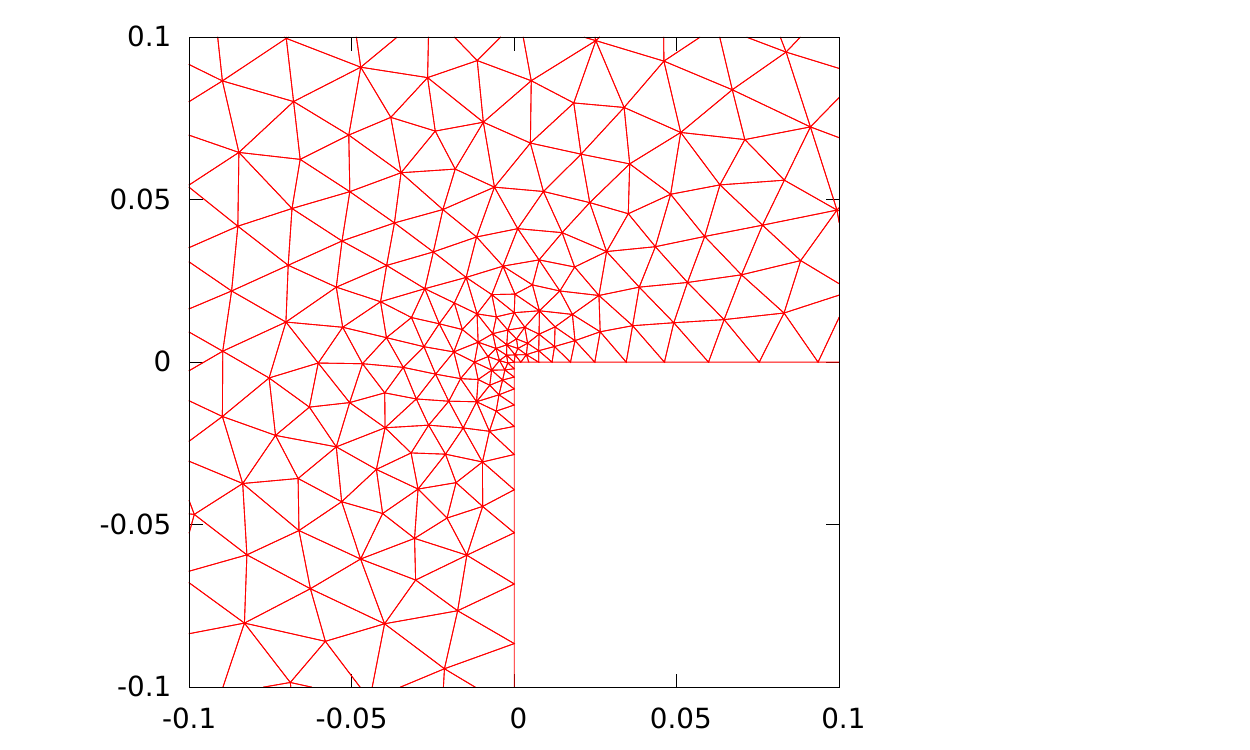}}
\end{minipage}\par\medskip
\caption{For the example in \S\ref{exam:L-shape}, a typical adaptive mesh and its close-up view near
the re-entrant corner are obtained with the metric tensor (\ref{M-1}).}
\label{fig:mesh08}
\end{figure}

\section{Effects of approximation of curved boundaries on eigenvalue computation}
\label{SEC:boundary}

In the previous sections we have assumed that $\Omega$ is a polygonal or polyhedral domain.
We now consider the case of a domain with curved boundaries and study the effects of boundary approximation
on eigenvalue computation. To be specific, we only consider the 2D situation.
Discussion for higher dimensional cases is similar.

Let $\Omega$ be a bounded domain with a piecewise smooth (more precisely, H\"{o}lder class $C^{0,1}$)
boundary $\partial \Omega$. We consider the approximation of $\Omega$ by a polygonal domain
$\Omega_h$. We assume that $\partial \Omega_h$ is made up by a sequence of points $\{ P_i, i = 0, ..., N_b\}$
that satisfies (a) $P_0 = P_{N_b}$; (b) $d(P_i, P_{i+1}) \le h_b$, $i = 0, ..., N_b-1$, where $d(P_i, P_{i+1})$
is the distance between $P_i$ and $P_{i+1}$; and (c) $\{ P_i \}$ includes all of the corner points of $\partial \Omega$.
It can be shown that there exists a constant $\gamma > 0$ such that
\[
\max\limits_{P \in \partial \Omega_h} \min\limits_{Q \in \partial \Omega} d (P, Q) \le \gamma h_b^2 . 
\]

The effects of boundary approximation on the numerical solution of BVP of elliptic operators
have been studied extensively in the past. For example,  Strang and Berger \cite{Strang-Berger1973}
show that if $\Omega$ is convex and the solution ($u$) to a homogeneous Dirichlet BVP
 is sufficiently smooth, then
\begin{equation}
\| u - u_{\Omega_h}\|_{L^2(\Omega_h)} \le C h_b^2,\quad
\| \nabla (u - u_{\Omega_h})\|_{L^2(\Omega_h)} \le C h_b^{\frac{3}{2}},
\label{SB-1}
\end{equation}
where $C$ is a constant depending on $u_{\Omega}$ and $u_{\Omega_h}$ is
the solution to the same BVP but on $\Omega_h$. As a consequence, we have
\begin{equation}
\| \nabla (u- u^h)\|_{L^2(\Omega_h)} \le C_1 h_b^{\frac{3}{2}} + C_2 h,
\label{SB-2}
\end{equation}
where $u^h$ denote the linear finite element solution of the BVP defined on $\Omega_h$
and $h$ is the maximal element diameter of the mesh.

On the other hand, for eigenvalue problems in the form of (\ref{eigen-1})
Burenkov and Lamberti \cite{BuLa2008} show that
\[
\frac{| \lambda_j^{\Omega_h} - \lambda_j |}{\lambda_j} \le C h_b^2, \quad j = 1, 2, ...
\]
where $\lambda_j^{\Omega_h}$ and $\lambda_j$ denote the $j^{\text{th}}$ eigenvalues of (\ref{eigen-1})
defined on $\Omega_h$ and $\Omega$, respectively, and $C$ is a constant independent of $j$, $h_b$, and
eigenvalues. Recall that $\lambda_j^h$ is the linear finite element approximation of $\lambda_j^{\Omega_h}$
on polygonal domain $\Omega_h$. Using Theorem~\ref{thm-err} (with $d = 2$), we get 
\begin{equation}
\frac{|\lambda_j^h - \lambda_j|}{\lambda_j^h} \le C_1 h_b^2 + C_2 N^{-1},\quad j = 1, ..., k
\label{eigen-err-1}
\end{equation}
where $C_1$ and $C_2$ are constants independent of $j$, $h_b$, and $N$.
This estimate has an important implication in practical computation. Indeed, adaptive mesh computation
typically starts with a coarse initial mesh which is also used to define the geometry of the domain.
Although many mesh generation software can reconstruct the curved boundaries from the boundary
points of the initial mesh during the process of mesh refinement, there is no guarantee that the approximation
of the boundary is good enough such that the first term in (\ref{eigen-err-1}) is comparable to or smaller than
the second term. A way to avoid this problem is to use an initial mesh with a sufficient number
of boundary points such that $h_b = \mathcal{O}(N^{-\frac{1}{2}})$. 
On the other hand, from (\ref{LB-2}) we see that the condition for BVPs is
$h_b = \mathcal{O}(h^{\frac{2}{3}}) = \mathcal{O}(N^{-\frac{1}{3}})$.
Thus, the computation of eigenvalue problems requires more initial boundary mesh points than that for BVPs
in order to reduce the effects of the boundary approximation to the level of the error of the finite element approximation.

To verify (\ref{eigen-err-1}), we consider an example in the form of (\ref{eigen-1}), with $\bD = I$, $\rho = 1$,
and $\Omega$ being a circular sector with the central angle of $3 \pi /2$. The eigenvalues and eigenfunctions
are known (cf. \cite{WuZh2009}) to be
\[
\lambda_{j}=\alpha_{j}^{2},\quad
u_{j}=\frac{v_{j}}{\|v_{j}\|_{L^{2}(\Omega)}},
\quad v_{j}=J_{2m_{j}/3}(\alpha_{j}r)\sin(2m_{j}\theta/3),
\quad j = 1, 2, ...
\]
where $m_{j}$ is an integer depending on $j$ (with $m_j = 1,2,3,4,1,5,2,6, ...$)
and $\alpha_{j}$ is a zero of the Bessel function $J_{2m_{j}/3}$. 
The eigenfunction $v_{j}$ has the same singularity as $r^{2m_j/3}\sin(2m_j \theta/3)$.
We compute this problem using the anisotropic mesh adaptation strategy described in \S\ref{SEC:ama}.
The convergence behavior of the error in the eigenvalues and the resultant adaptive meshes are similar to
those for the L-shape example considered in \S\ref{exam:L-shape}. For this reason, we omit those results
here to save space. Instead, we present a result obtained with a mesh of about 75,000 elements
(so the second term in (\ref{eigen-err-1}) can be ignored when compared to the first term) but starting with
initial meshes of different numbers of boundary points (which are distributed almost evenly
along the arc of the domain). In Fig.~\ref{sperror} the relative error in the first four eigenvalues
is plotted as a function of the number of the boundary points in the initial mesh.
It is clearly shown that the error converges at a second order rate in terms of $h_b = \mathcal{O}(N_b^{-1})$,
which is consistent with (\ref{eigen-err-1}). The relative error for the first eigenvalue is also plotted
in Fig.~\ref{sperror30} as a function of the number of mesh elements $(N)$ for a fixed number of boundary
points in the initial mesh. One can see that the error decreases for small $N$
but stops decreasing for large $N$ when the boundary approximation error (the first term in (\ref{eigen-err-1}))
dominates the whole error.

\begin{figure}
\begin{minipage}{.5\linewidth}
\centering
\subfloat[]{\label{sperror:a}\includegraphics[scale=.9]{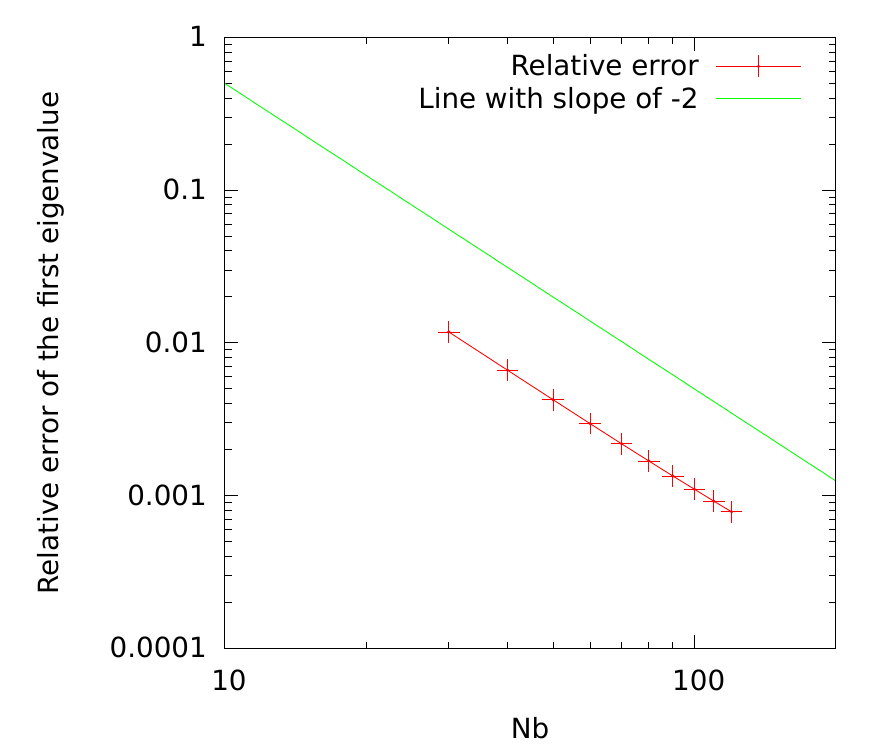}}
\end{minipage}%
\begin{minipage}{.5\linewidth}
\centering
\subfloat[]{\label{sperror:b}\includegraphics[scale=.9]{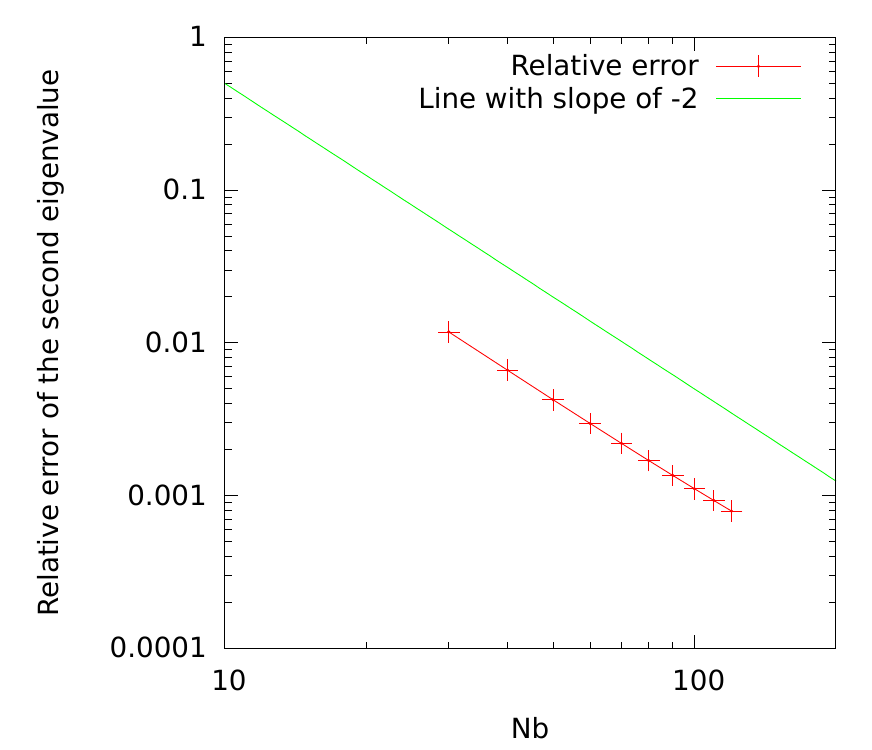}}
\end{minipage}\par\medskip
\begin{minipage}{.5\linewidth}
\centering
\subfloat[]{\label{sperror:c}\includegraphics[scale=.9]{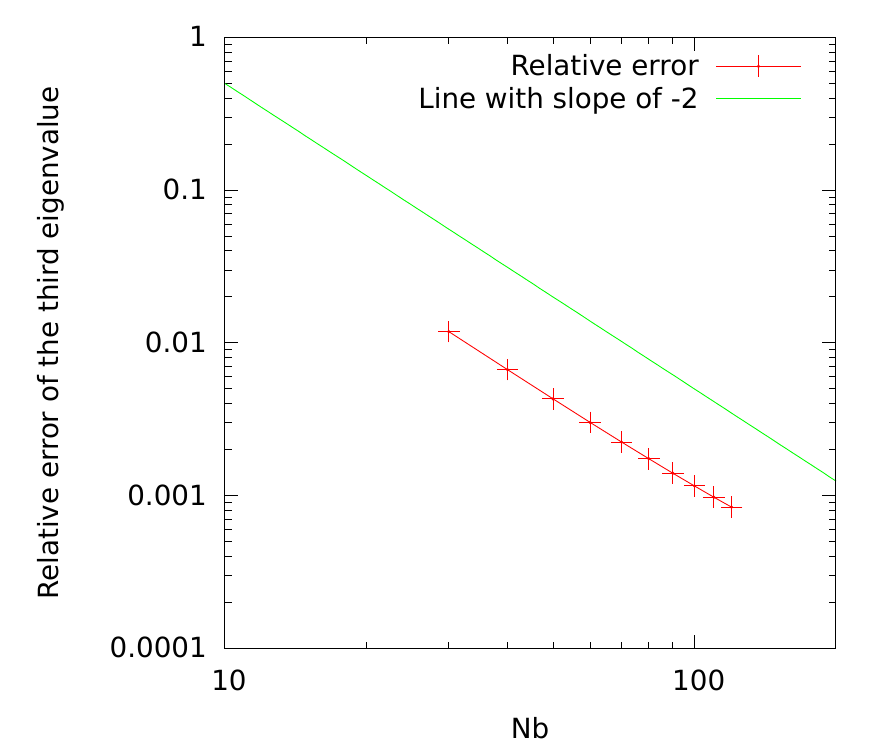}}
\end{minipage}
\begin{minipage}{.5\linewidth}
\centering
\subfloat[]{\label{sperror:d}\includegraphics[scale=.9]{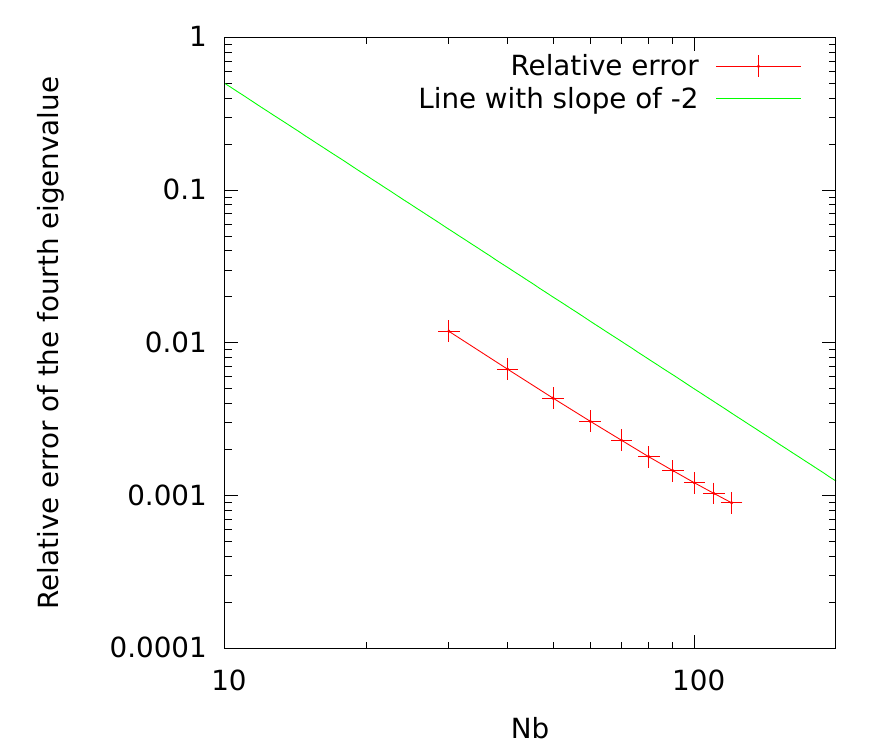}}
\end{minipage}
\caption{For the example in \S\ref{SEC:boundary}, the relative error in the first four eigenvalues
is plotted as a function of the number of the boundary points in the initial mesh.}
\label{sperror}
\end{figure}

\begin{figure}
\begin{minipage}{.5\linewidth}
\centering
\subfloat[]{\label{sperrorn30:a}\includegraphics[scale=.9]{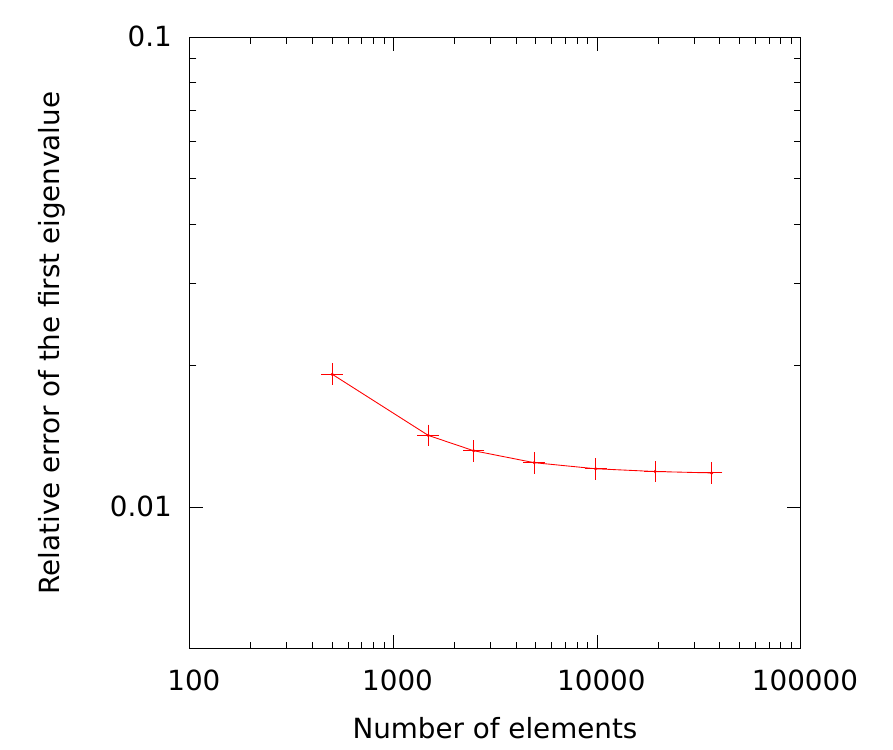}}
\end{minipage}%
\begin{minipage}{.5\linewidth}
\centering
\subfloat[]{\label{sperrorn30:b}\includegraphics[scale=.9]{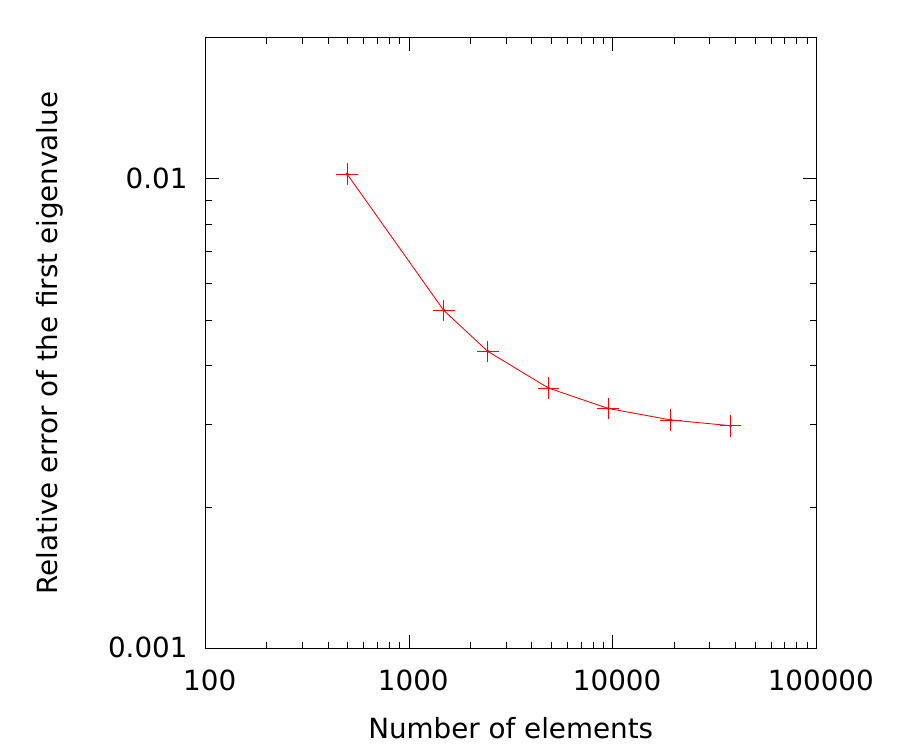}}
\end{minipage}
\caption{For the example in \S\ref{SEC:boundary}, the relative error in the first eigenvalue
is plotted as a function of the number of mesh elements for a fixed number of
boundary points in the initial mesh. (a) $N_b = 30$ and (b) $N_b = 60$.}
\label{sperror30}
\end{figure}

\section{Conclusions}
\label{SEC:conclusion}

In the previous sections we have studied an anisotropic mesh adaptation strategy
for the finite element computation of anisotropic eigenvalue problems in the form of (\ref{eigen-1}).
The study is based on the $\bM$-uniform mesh approach with which any nonuniform mesh
is characterized mathematically as a uniform one in the metric specified by a metric tensor defined
on the physical domain. Linear finite elements on simplicial meshes are considered.
Error bounds for the finite element approximations of the eigenvalues and eigenfunctions,
(\ref{error-3}) and (\ref{thm-err-1}) of Theorem~\ref{thm-err}, have been established
for quasi-$\bM$-uniform meshes associated with
the metric tensor (\ref{M-1}). It is worth mentioning that (\ref{M-1}) depends on the Hessian
of the eigenfunctions to be sought. In practice, this can be replaced by the approximate Hessian
obtained through Hessian recovery from function nodal values.
An iterative procedure for the adaptive mesh solution of eigenvalue problems is given in Fig.~\ref{f.1}.
Numerical results obtained for four two dimensional examples have demonstrated that the procedure
can produce properly adapted, anisotropic meshes which lead to more accurate computed eigenvalues
than uniform or isotropic adaptive meshes. They have also confirmed the second order convergence rate of the error
predicted in (\ref{thm-err-1}) as the mesh is being refined.

We have also studied the effects of boundary approximation on the computation of eigenvalue problems
with curved boundaries. It is shown that about $N^{1/2}$ boundary points in the initial
mesh, where $N$ is the number of the elements in the final adaptive mesh, are needed to keep
the effects of boundary approximation at the level of the error of the finite element approximation.
On the other hand, only about $N^{1/3}$ boundary points are needed in the initial mesh for boundary
value problems. This indicates that the computation of eigenvalue problems is more sensitive to boundary approximation
than that of boundary value problems.

\vspace{20pt}

\noindent
{\bf Acknowledgment.}
This work was supported in part by the NSF under Grant DMS-1115118.


\end{document}